\documentclass{article}
\usepackage{amsmath,amssymb,amsfonts,amsthm,amscd}

\newtheorem{theorem}{Theorem}[section]

\newtheorem{lemma}[theorem]{Lemma}
\newtheorem{corollary}[theorem]{Corollary}

\theoremstyle{definition}

\theoremstyle{remark}

\title{A complete classification of\\ three-dimensional algebras over ${\mathbb R}$ and ${\mathbb C}$ \\-- {\it OnKoChiShin} (visiting old, learn new)}

\author{Yuji Kobayashi, Kiyoshi Shirayanagi, Sin-Ei Takahasi and Makoto Tsukada}

\date{}

\begin{document}

\maketitle


\footnote[0]{Research of the first, third and fourth authors was supported in part by JSPS KAKENHI Grant Number JP25400120.}
\footnote[0]{Research of the second author was supported in part by JSPS KAKENHI Grant Number JP15K00025.}

\begin{abstract}
We provide a complete classification of three-dimensional associative algebras over the real and
complex number fields based on a complete elementary proof.
  We list up all the multiplication tables 
of the algebras up to isomorphism.  We compare our results with those given
by mathematicians in the 19th century.
\end{abstract}


\section{Introduction}
Since the discovery of the quaternions by Hamilton \cite{H}, 
many mathematicians had been trying to extend number systems to more general systems
such as 
tessarines, coquaternions, biquaternions, 
and algebras introduced by Grassmann and by Clifford etc.
(see Kolmogorov and Yushkevich \cite{K} and van der Waerden \cite{W}). 
These are finite-dimensional associative algebras over the real number field $\mathbb{R}$, 
which were called "hypercomplex number systems" at that time. 
It was Peirce who first studied associative algebras systematically (\cite{P}).  
He enumerated "pure" algebras of low dimension using so called Peirce's decomposition.  
However, not only was his terminology ambiguous and were his proofs not rigorous,  but
he also did not give proofs for non-isomorphism between the listed algebras.
Hawkes \cite{H2} and Taber \cite{T} tried to give rigid reasoning for the methods of Peirce, 
but their resulting proofs are not satisfactory from the viewpoint of modern mathematics.
On the other hand, Scheffers \cite{Sch} and Study \cite{St} determined unital algebras of low dimension
over $\mathbb{R}$ and the complex number field $\mathbb{C}$.

In this paper, we classify (associative but not necessarily unital) 
algebras of dimension 3 over $\mathbb{R}$ and $\mathbb{C}$ up to isomorphism.  
Our classification is complete and the methods 
are very different from Peirce's. The proofs are rigorous and concrete.  
We first treat unital algebras in a classical way. Then, we divide non-unital algebras 
into three types, {\it curled}, {\it waved} and {\it straight} algebras.
An algebra $A$ is curled if $x$ and $x^2$ are linearly dependent 
for any $x \in A$, it is waved if it is not curled but $x, x^2$ and $x^3$ are
linearly dependent for any $x \in A$, and it is straight otherwise. 
We look into subalgebras of A of dimension 2, and consider all combinations of them 
to list up all possible multiplication tables of algebras of dimension 3.  Finally, we make
a complete list of algebras of dimension 3 that are not isomorphic to each other. 
Our methods are elementary and we use ideal theory as little as possible.

We appreciate the researches of the abovementioned past mathematicians,
in particular, we respect Peirce for his pioneering work. 
By presenting our new complete classification, we attempt to make the paper as an {\it homage} to his work.  
We believe that the spirit of this paper can be best expressed in the subtitle, consisting of four Chinese characters which are pronounced as "OnKoChiShin".
It is a word from the Analects of Confucius \cite{Wat}, which is commonly used in Japan
as an idiom meaning ``visiting old, learn new", or more precisely,
discovering new things by studying the past through scrutiny of the old.

Our result is as follows.
Over $\mathbb{C}$, up to isomorphism, we have 
5 unital algebras $$U^3_0, U^3_1, U^3_2, U^3_3, U^3_4,$$ and among the non-unital algebras, we have 
5 curled algebras $$C^3_0, C^3_1, C^3_2, C^3_3, C^3_4,$$
4 straight algebras $$S^3_1, S^3_2, S^3_3, S^3_4$$  
and 9 waved algebras $$W^3_1, W^3_2, W^3_4, W^3_5, W^3_6, W^3_7, W^3_8, W^3_9, W^3_{10}$$
and one infinite family $$\big\{W^3_3(k)\big\}_{k\in \mathcal{H}}$$ of waved algebras, where the symbols $U^3_i$, $C^3_i$, $S^3_i$ and $W^3_i$ denote specified
multiplication tables and $\mathcal{H}$ is the right half-plane $\{x+y{\bf i} \,|\, x>0 \ \mbox{or}\ x=0, y\geq 0\}$ of the complex plane.
Over $\mathbb{R}$, in addition to the above algebras, we have one unital algebra 
$U^3_{2^-}$, one non-unital straight algebra $S^3_{3^-}$ and one infinite family 
$\big\{W^3_{3^-}(k)\big\}_{k \geq 0}$ of non-unital waved algebras. The concrete description of the above multiplication tables is given in Section~\ref{sec:summary}.

The rest of this paper is organized as follows. First, we give a classification of two-dimensional algebras in Section~\ref{sec:2dim}, 
because it will be utilized to achieve the aim of classifying three-dimensional algebras.
Next, we determine unital algebras in Section~\ref{sec:unital}.  
We classify curled algebras in Section~\ref{sec:curled} and straight algebras in Section~\ref{sec:straight}.
In the subsequent three sections, we treat the most difficult case where algebras are waved. 
In this case we focus on two subalgebras having their one-dimensional intersection.
We explain the strategy for the classification in Section~\ref{sec:w-strategy}, 
enumerate all possible waved algebras in Section~\ref{sec:w-enumerate},
and check if the algebras enumerated are isomorphic or not in Section~\ref{sec:w-screen}.  
In the final section, we summarize the results and give a complete list of algebras 
modulo isomorphism, and compare it with the lists given by Peirce, Scheffers and Study.

Throughout this paper, we assume that an algebra is associative.

\section{Preliminaries}\label{sec:pre}
Let $K = \mathbb{R}$ or $K = \mathbb{C}$.
In this paper, we express an algebra over $K$ by the multiplication table.
Let $A$ be an algebra over $K$ of dimension $n$ where $\{e_1,\dots,e_n\}$ is a linear bases of $A$. 
Then, the structure of $A$ is determined by the table which is an $n\times n$ matrix
\begin{equation}\label{eq:structure}
\begin{pmatrix}
a_{11} & \cdots & a_{1n} \\
\vdots & \ddots & \vdots \\
a_{n1} & \cdots & a_{nn}
\end{pmatrix}
\end{equation}
with $a_{ij}\in A$, 
where $a_{ij}=e_i e_j$ for all $i\in [1,n]$ and $j\in [1,n]$.
That is,
for $p = \sum_{i=1}^n x_ie_i$, $q = \sum_{i=1}^n y_ie_i\in A$ with
$x_i, y_i \in K$,
the product $pq$ is defined by
$$ pq \  = \sum_{i,j\in [1,n]} x_iy_ja_{ij}. $$
We say that $A$ is the algebra on $\{e_1,\dots,e_n\}$ defined by (\ref{eq:structure}).

We refer to $A$ as a {\it zero algebra} if its multiplication table is the zero matrix with $a_{ij}=0$ for all $i\in [1,n]$ and $j\in [1,n]$, or
equivalently, if $pq=0$ for all $p, q\in A$.
Related to zero algebras, we call $A$ a {\it zeropotent algebra} if $p^2=0$ for all $p\in A$. Obviously a zero algebra is zeropotent, but in general the converse is not true.
In fact, it is easy to show that
$A$ is zeropotent if and only if $A$ is defined by an alternative matrix with $a_{ij}=-a_{ji}$ for all $i\in [1,n]$ and $j\in [1,n]$.
For a classification of zeropotent {\it nonassociative} algebras, see \cite{KSTT,STTK}.


Let us describe a criterion for two algebras to be isomorphic. 
Let $A$ and $B$ be algebras over $K$ of dimension $n$ defined by tables ($n\times n$ matrices), $(a_{ij})$ and $(b_{ij})$, 
where $\{e_1,\dots,e_n\}$ and $\{f_1,\dots,f_n\}$ are linear bases of $A$ and $B$, respectively. 
Moreover, let $e_i e_j=a_{ij}=\sum_{s=1}^n\alpha_{ij}^s e_s$ and 
$f_i f_j=b_{ij}=\sum_{t=1}^n\beta_{ij}^t f_t$ for all $i\in [1,n]$ and $j\in [1,n]$.
Suppose that $\varphi: A\rightarrow B$ is an algebra homomorphism and $M_\varphi=(m_{st})$ with $m_{st}\in K$
is the $n\times n$ matrix of $\varphi$ as a linear mapping with respect to the linear bases:
\[\varphi(e_s)=\sum_{t=1}^n m_{st}f_t.\]
Let us call $M_\varphi$ a transformation matrix from $A$ to $B$.
For each $(i,j)$, the equality $\varphi(e_i)\varphi(e_j)=\varphi(e_ie_j)=\varphi(a_{ij})$ holds. 
Calculating the both sides, we see that $A$ is isomorphic to $B$ over $K$ if and only if 
there exists a non-singular matrix $M_{\varphi}$ such that
\[\sum_{s}\alpha_{ij}^s\,m_{st}=\sum_{k,l}\beta_{kl}^t\,m_{ik}m_{jl}\mbox{ for all $t,i,j\in [1,n]$.} \,\,\cdots\cdots\cdots \mbox{($\ast$)}\]

Assume that $A$ is an algebra over $K$ of dimension 3.
An element $e \in A$ is {\it curled} (resp.\,{\it waved}) if
$\{e, e^2\}$ (resp. $\{e, e^2, e^3 \}$) is linearly dependent over $K$. 
An algebra $A$ is called {\it curled} if every element of $A$ is curled.
$A$ is called {\it waved} if $A$ is not curled but every element of $A$ is waved.
Note that in a waved algebra there is an element $e$ such that $\{e,e^2\}$ is linearly independent
but every such element $e$ satisfies that $\{e,e^2,e^3\}$ is linearly dependent.
We call $A$ {\it straight}, if it is not curled nor waved, that is, 
if $\{e, e^2, e^3\}$ forms a linear base of $A$ for some $e \in A$.

When $B$ is an algebra of dimension 2, $B$ is {\it curled} 
if every element of $B$ is curled, otherwise $B$ is {\it straight}.

For two algebras $A$ and $B$ over $K$, 
$A\oplus B$ denotes their direct sum; 
for $p = (x,y), q = (x',y') \in A\oplus B$ their product is defined by
$pq = (xx', yy')$. 

\section{Algebras of dimension 2}\label{sec:2dim}
Before entering the discussion on two-dimensional algebras, let us consider one-dimensional algebras.
Let $A$ be an algebra over $K$ of dimension 1 with a linear base $\{e\}$.
Then, $e^2=ke$ for some $k\in K$. Depending on whether $k$ is 0 or not, it follows that $A$ is the algebra defined by $e^2=0$ or is isomorphic to the algebra defined by $e^2=e$.
We denote the former algebra by $A^1_0$ and the latter by $A^1_1$. $A^1_0$ is a zero algebra and $A^1_1$ is a unital algebra with the identity element $e$.

Now let $A$ be an algebra over $K$ of dimension 2 with a linear base $\{e, f\}$.

First, suppose that $A$ is curled.  Then, 
\begin{equation}\label{eq:curled}
e^2 = ke \ \, \mbox{and} \ \, f^2 = \ell f
\end{equation}
for some $k, \ell \in K$.  Here if $k \neq 0$ (resp. $\ell \neq 0$), 
replacing $e$ (resp. $f$) by $e/k$ (resp. $f/\ell$),
we may assume that $k$ and $\ell$ are 0 or 1 in (\ref{eq:curled}).

We can write
$$ ef = ae + bf \ \, \mbox{and} \ \, fe = ce + df $$
with $a, b, c, d \in K$.
We have
$$ kae + kbf = kef = e^2f = e(ef) = e(ae + bf) = ae^2 + bef = a(k + b)e + b^2f. $$
It follows that
\begin{equation}\label{eq:ab}
ab = 0 \ \, \mbox{and} \ \, b^2 = kb.
\end{equation}
Similarly we have
\begin{equation}\label{eq:cd}
cd = 0, \, a^2 = \ell a, \, c^2 = \ell c \ \, \mbox{and} \ \, d^2 = k d. 
\end{equation}
We have
$$ kce + ade + bdf = e(ce+df) = efe = (ae + bf)e = kae + bce + bdf.  $$
It follows that
\begin{equation*}
k(a - c) = ad - bc.
\end{equation*}
Similarly we have
\begin{equation*}
\ell(d - b) = ad - bc.
\end{equation*}
Because $A$ is curled,
$xe + yf$ and
$$
\begin{array}{lll}
(xe + yf)^2 & = & x^2e^2 + xyef + xyfe +y^2f^2 \\
& = & (kx^2 + axy + cxy)e + (\ell y^2 + bxy + dxy)f. 
\end{array}
$$
are linearly dependent over $K$ for any $x, y \in K$.
Hence,
$$
\begin{vmatrix}
x & kx^2+(a+c)xy \\
y & \ell y^2+(b+d)xy
\end{vmatrix}
= \big((b+d-k)x + (\ell-a-c)y\big)xy = 0.
$$
It follows that
\begin{equation}\label{eq:kl}
k = b + d \ \, \mbox{and} \ \, \ell = a + c.
\end{equation}

It is easy to solve the equations (\ref{eq:ab}), (\ref{eq:cd}) and (\ref{eq:kl}) for 
$k, \ell \in \{0, 1\}$ and $a, b, c, d \in K$.  We 
obtain 7 solutions:
\begin{multline*}
 (k,\ell,a,b,c,d) \,= \,
(0,0,0,0,0,0), \,(0,1,0,0,1,0), \,(0,1,1,0,0,0), \\
(1,0,0,0,0,1), \,(1,0,0,1,0,0), \,(1,1,0,1,1,0),\,(1,1,1,0,0,1). 
\end{multline*}  
In correspondence to them
we have 7 curled algebras 
$A^2_0, A^2_1, A^2_2, A'^2_3, A'^2_4, A'^2_5$ and $A'^2_6$ defined by
$$ \begin{pmatrix}
0 & 0 \\
0 & 0
\end{pmatrix},
\begin{pmatrix}
0 & 0 \\
e & f
\end{pmatrix},
\begin{pmatrix}
0 & e \\
0 & f
\end{pmatrix},
\begin{pmatrix}
e & 0 \\
f & 0
\end{pmatrix},
\begin{pmatrix}
e & f \\
0 & 0
\end{pmatrix},
\begin{pmatrix}
e & f \\
e & f
\end{pmatrix}
\ \mbox{and} \ 
\begin{pmatrix}
e & e \\
f & f
\end{pmatrix},
$$
respectively.
Using the isomorphism criterion in the previous section, the algebras $A^2_1$ and $A'^2_4$
are isomorphic via the transformation matrix
$M_\varphi=\begin{pmatrix}
0 & 1\\
1 & 0
\end{pmatrix}$ which is a non-singular matrix satisfying ($\ast$).
Similarly, $A^2_2$ and $A'^2_3$ are isomorphic.
The algebra $A'^2_5$ is isomorphic to $A^2_1$
via the transformation matrix 
$\begin{pmatrix}
1 & 1\\
0 & 1
\end{pmatrix}$.
Similarly, $A'^2_6$ is isomorphic to $A^2_2$.
On the other hand, it is easy to see that there is no non-singular matrix satisfying  ($\ast$) between $A^2_1$ and $A^2_2$.
Consequently
we have three non-isomorphic curled algebras $A^2_0, A^2_1$, 
and $A^2_2$.

Next, we suppose that $A$ is straight, that is, there is $g \in A$ such that
$\{g, g^2\}$ is a linear base of $A$.  Then,
$A$ is commutative, and we can write
\begin{equation}\label{eq:g-cube}
g^3 = bg^2 + cg
\end{equation}
with $b, c \in K$.

If $b = c = 0$, that is, $g^3 = 0$, then letting $e = g^2$ and $f = g$, 
$A$ is defined by the table
\begin{equation}\label{eq:000e}
\begin{pmatrix}
0 & 0 \\
0 & e
\end{pmatrix}.
\end{equation}
Next, if $b \neq 0$ and $c = 0$, then
letting $f = g/b$ and $e = f^2$, $A$ is defined by the table
\begin{equation}\label{eq:all-e}
\begin{pmatrix}
e & e \\
e & e
\end{pmatrix}. 
\end{equation}
Lastly, suppose that $c \neq 0$.
Let 
\begin{equation}\label{eq:e}
e = \frac{1}{c}(g^2 - bg). 
\end{equation}
Then, using (\ref{eq:g-cube}), we see  $eg = ge = g$, and so 
$e$ is the identity element of $A$.  Moreover, by (\ref{eq:e}), 
$\{g, e\}$ is a linear base of $A$ and
\begin{equation}\label{eq:g-sq}
g^2 = bg + ce
\end{equation}
holds.  Let $ D = b^2 + 4c$, and let $h = 2g - be$.
Then, by (\ref{eq:g-sq}) we have
\begin{equation}\label{eq:h-sq}
h^2 = 4g^2 - 4bg + b^2e = De. 
\end{equation}
If $K = \mathbb{C}$ and $D \neq 0$, or $K = \mathbb{R}$ and $D > 0$, let
$ f = \frac{1}{\sqrt{D}}h$.
Then, by (\ref{eq:h-sq}) we have
$ f^2 = e$.
Hence, $A$ is defined by
\begin{equation}\label{eq:effe}
\begin{pmatrix}
e & f \\
f & e
\end{pmatrix}. 
\end{equation}
If $K = \mathbb{R}$ and $D < 0$, let $f = \frac{1}{\sqrt{-D}}h$.
Then, 
$f^2 = -e$ by (\ref{eq:h-sq}).
Hence, $A$ is defined by
\begin{equation}\label{eq:-e}
\begin{pmatrix}
e & f \\
f & -e
\end{pmatrix}. 
\end{equation}
If $D = 0$, let $ f = h$,
then, $f^2 = 0$ by (\ref{eq:h-sq}), and $A$ is defined by
\begin{equation}\label{eq:eff0}
\begin{pmatrix}
e & f \\
f & 0
\end{pmatrix}.
\end{equation}

If $A$ is defined by the table in (\ref{eq:all-e}), then by replacing $f$ by $f-e$, 
$A$ is defined also by 
\begin{equation}\label{eq:e000}
\begin{pmatrix}
e & 0 \\
0 & 0
\end{pmatrix}.
\end{equation}
If $A$ is defined by the table in (\ref{eq:effe}), then by replacing $e$ by $\frac{e+f}{2}$ and $f$ by $\frac{e-f}{2}$,
$A$ is defined by
\begin{equation}\label{eq:ef}
\begin{pmatrix}
e & 0 \\
0 & f
\end{pmatrix}.
\end{equation}

We denote the algebras defined by the tables in (\ref{eq:000e}), (\ref{eq:e000}), (\ref{eq:effe}), (\ref{eq:-e})
and (\ref{eq:eff0}) by $A^2_3, A^2_4, A^2_5, A^2_{5^-}$ and $A^2_6$,
respectively.
Using the isomorphism criterion in Section~\ref{sec:pre} it is easy to 
show that these algebras are not isomorphic to each other.

\section{Unital algebras}\label{sec:unital}
In this section, $A$ is a unital algebra 
over $K$ of dimension 3 with identity element 1.

First, suppose that there is an element $h$ of $A$ such that
$\{1, h, h^2\}$ forms a linear base of $A$, 
in this case we 
say that $A$ is {\it unitally straight}.
Then, $A$ is commutative, and we can write
\begin{equation*}
h^3 = ah^2 + bh + c
\end{equation*}
with $a, b, c \in K$.   Thus, the algebra $A$ is generated by $h$ and
is isomorphic to the residue algebra
of the polynomial algebra $K[X]$ modulo the ideal generated by
$$ P(X) = X^3- aX^2 - bX - c.$$  
Let $\alpha, \beta$ and $\gamma$ be 
the roots of $P$ in ${\mathbb C}$;
$$ P(X) = (X - \alpha)(X - \beta)(X -\gamma). $$

(i)  Suppose that $\alpha, \beta$ and $\gamma$ are different 
from each other.

(i1)  Suppose that $\alpha, \beta, \gamma \in K$.
We have an isomorphism 
$ \phi: A \rightarrow K \oplus K \oplus K$
of algebras defined by
$$ \phi(h) = (\alpha, \beta, \gamma). $$
Let $e = \phi^{-1}(1,0,0)$, $f = \phi^{-1}(0,1,0)$ 
and $g = \phi^{-1}(0,0,1)$.  Then, on the base $\{e,f,g\}$, 
$A$ is defined by
\begin{equation}\label{eq:efg}
\begin{pmatrix}
e & 0 & 0\\
0 & f & 0\\
0 & 0 & g
\end{pmatrix}.
\end{equation} 

(i2)  Next, suppose that $K = \mathbb{R}$, and $\alpha$ is real but 
$\beta$ and $\gamma$ are not real.  We have an isomorphism
$\phi: A \rightarrow K \oplus B$ of algebras defined by
$$ \phi(h) = (\alpha, \bar{h}), $$
where $B = K[X]/((X-\beta)(X-\gamma))$ and 
$\bar{h}$ is the image of $h$
by the natural surjection 
from 
$A$ to $B$.  By the results in the previous
section, $B$ is defined by 
$\begin{pmatrix}
f' & g' \\
g' & -f'
\end{pmatrix}$ of $B$
on a suitable base $\{f', g'\}$.  Let $e = \phi^{-1}(1,0)$
$f = \phi^{-1}(0,f')$ and $g = \phi^{-1}(0,g')$, then
on the base $\{e,f,g\}$, $A$ is defined by
\begin{equation}\label{eq:e000fg0g-f}
\begin{pmatrix}
e & 0 & 0\\
0 & f & g\\
0 & g & -f
\end{pmatrix}.
\end{equation} 

(ii)  Next, suppose that $\alpha \neq \beta = \gamma$, then
$\alpha, \beta \in \mathbb{R}$ and
$$P(X) = (X - \alpha)(X - \beta)^2. $$
We have an isomorphism $\phi: A \rightarrow K \oplus B$
defined by
$ \phi(h) = (\alpha, \bar{h}),$
here $B = K[X]/((X-\beta)^2)$ and $\bar{h}$ is the natural image of $h$ 
in $B$.  By the results in the previous section $B$ is defined by
$\begin{pmatrix}
f' & g' \\
g' & 0
\end{pmatrix}$
on the base $\{f' , g'\}$ with $f '= 1$ and $g' = \overline{h}-\beta$.
Thus, in the same way as above $A$ is defined by
\begin{equation}\label{eq:e000fg0g0}
\begin{pmatrix}
e & 0 & 0\\
0 & f & g\\
0 & g & 0
\end{pmatrix}.
\end{equation} 
on a suitable base $\{e,f,g\}$.

(iii)  Lastly, suppose that $\alpha = \beta = \gamma$, that is, 
$P(X) = (X-\alpha)^3$.  Let $e = 1$,
$f = h - \alpha$ and $g = f^2. $
Then, $fg = gf = g^2 = 0$, and so $A$ is defined by
\begin{equation}\label{eq:efgfg0g00}
\begin{pmatrix}
e & f & g \\
f & g & 0 \\
g & 0 & 0
\end{pmatrix}.
\end{equation}
 
Using the isomorphism criterion it is easy to show that the algebras defined by (\ref{eq:efg}), (\ref{eq:e000fg0g0}) and (\ref{eq:efgfg0g00}) 
are not isomorphic to each other.  When $K = \mathbb{R}$,
the algebra defined by (\ref{eq:e000fg0g-f}) is not isomorphic to the algebras above either.
  
Next, suppose that $A$ is not unitally straight, that is, 
there is no $h\in A$ such that $\{1,h,h^2\}$ forms a linear base of $A$.
Let $\{1, f, g\}$ be a linear base of $A$.
Then $\{1,f\}$ is a linear base of the unital subalgebra $B$ of $A$ 
generated by $f$.  
As discussed in the previous section, $B$ is defined by (\ref{eq:effe}) or (\ref{eq:eff0}) 
when $K = \mathbb{C}$.  Hence, by changing $f$ by a suitable element 
of $B$, we may suppose that $f^2 = 0$ or $f^2 = 1$.  
When $K = \mathbb{R}$, there is another possibility $f^2 = -1$, 
because $B$ may be defined by (\ref{eq:-e}).
Similarly, we may suppose that $g^2 = 0$ or $g^2 = 1$ (or $g^2 = -1$
when $K = \mathbb{R}$).

Let 
$fg = a + bf + cg$ \, and \, $gf = a' + b'f + c'g$ 
\, for $a, b, c, a', b', c' \in K$.
If $f^2 = 0$, then we have 
$$ 0 = f^2g = f(a + bf + cg) = af + c(a+bf+cg) = ac + (a+bc)f + c^2g. $$ 
and
$$ 0 = gf^2 = (a'+b'f+c'g)f = a'f + c'(a'+b'f+c'g) 
= a'c' + (a'+b'c')f+ c'^2g. $$
It follows that
\begin{equation}\label{eq:a=c}
a = c = a' = c' = 0.
\end{equation}
Similarly, if $g^2 = 0$, we have
\begin{equation}\label{eq:a=b}
a = b = a' = b' = 0.
\end{equation} 

If $f^2 = 1$, then
$$ g = f^2g = f(a+bf+cg) = af+b+c(a+bf+cg) = ac+b + (a+bc)f + c^2g. $$
Thus we have 
$$ ac + b = a + bc = c^2 - 1 = 0. $$
Similarly, we have
$$ a'c' + b' = a' + b'c' = c'^2 - 1 = 0.$$
Hence, we have
\begin{equation}\label{eq:c=pm1}
c = \pm 1, b = \mp a \ \mbox{(double-sign corresponds)} \ \ \mbox{and} 
\ \ c' = \pm 1, b' = \mp a' \ \mbox{(d-s.c.)}.
\end{equation}
Similarly, if $g^2 = 1$, we have
\begin{equation}\label{eq:b=pm1}
b = \pm 1, c = \mp a \ \mbox{(d-s.c.)}\ \ \mbox{and} 
\ \ b' = \pm1, c' = \mp a' \ \mbox{(d-s.c.)}.
\end{equation}

Finally, if $K = \mathbb{R}$ and $f^2 = -1$, then we have
$$ -g = f^2g = f(a + bf + cg) = af - b + c(a + bf + cg) =
ac - b + (a + bc)f +  c^2g.  $$
Hence, $c^2 = -1$, but this is impossible in $\mathbb{R}$.
Similarly, $g^2 = -1$ is impossible.

(i)  Suppose that $f^2 = g^2 = 0$, then by (\ref{eq:a=c}) and (\ref{eq:a=b})
we have $ a = b = c = a' = b' = c' = 0$.
Hence, $A$ is defined by
\begin{equation}\label{eq:efgf00g00}
\begin{pmatrix}
e & f & g \\
f & 0 & 0 \\
g & 0 & 0
\end{pmatrix}
\end{equation}
on the base $\{e, f, g\}$ with $e = 1$.

(ii)  If $f^2 = 0 $ and $g^2 = 1$, then by (\ref{eq:a=c}) and (\ref{eq:b=pm1})
$$ a = c = a' = c' = 0, \ b = \pm 1 \ \ \mbox{and} \ \ b' = \pm 1. $$
Here, if $b = b' = 1$, that is, $fg = gf = f$, then $1, f +g$ and
$$ (f + g)^2 = f^2 + fg + gf + g^2 =  2f + 1 $$ 
are linearly independent.  Hence, $A$ is unitally straight.
Similarly, the case $b= b' = -1$ gives a unitally straight algebra.

If $b = 1$ and $b' = -1$, then $A$ is defined by
\begin{equation}\label{eq:efgf0fg-fe}
\begin{pmatrix}
e & f & g \\
f & 0 & f \\
g & -f & e
\end{pmatrix} \end{equation}
with $e = 1$.
This is not unitally straight, because $1, x + yf + zg$ and
$$(x+yf+zg)^2 = x^2 + z^2 + 2x(yf + zg) $$
are linearly dependent for any $x, y, z \in K$.

The case $b = -1$ and $b' = 1$ gives the algebra isomorphic to 
the previous algebra 
defined by (\ref{eq:efgf0fg-fe})
via the transformation matrix
$\begin{pmatrix}
1 & 0 & 0\\
0 & -1 & 0 \\
0 & 0 & -1
\end{pmatrix}$.

(iii)  The case $f^2 = 1, g^2 = 0$ is symmetric to (ii), and yields no new algebra.

(iv)  Suppose that $f^2 = g^2 = 1$.
Because $A$ is not unitally straight, $1,\, h = x + yf + zg$ and 
$$\begin{array}{lll}
h^2 & = & x^2 + y^2 + z^2 + 2xyf + 2xzg + yz(fg+gf) \\
& = & x^2+y^2+z^2+yz(a+a') + y(2x+z(b+b'))f + z(2x+y(c+c'))g
\end{array}$$
are linearly dependent for any $x, y, z \in \mathbb{R}$.
Hence,
$$\begin{vmatrix}
1 & x & x^2+y^2+z^2+yz(a+a') \\
0 & y & y(2x+z(b+b')) \\
0 & z & z(2x+y(c+c'))
\end{vmatrix} 
= yz(y(c+c') - z(b+b')) = 0$$
for any $x,y,z \in \mathbb{R}$.  It follows that
\begin{equation}\label{eq:b+b'}
b + b' = c + c' = 0.
\end{equation}

We have 4 solutions in $a, b, c, a', b', c'$ satisfying (\ref{eq:c=pm1}), (\ref{eq:b=pm1}) and (\ref{eq:b+b'}):
$$\begin{array}{lll}
(a, b, c, a', b', c') & = & 
(-1,1,1,-1,-1,-1), \ 
(1,-1,1,1,1,-1), \\
& & (1,1,-1,1,-1,1), \ 
(-1,-1,-1,-1,1,1). 
\end{array}$$
 
Now, if $(a,b,c,a',b',c') = (-1,1,1,-1,-1,-1)$, that is, 
$$ fg = -1 + f + g \ \ \mbox{and} \ \ gf = -1 - f - g, $$ 
then let
$f' = f + g$ and $g' = g$.
If $(a,b,c,a',b',c') = (1,-1,1,1,1,-1)$, let
$f' = f - g$ and $g' = -g$.
If $(a,b,c,a',b',c') = (1,1,-1,1,-1,1)$, let
$f' = f - g$ and $g' = g$.
If $(a,b,c,a',b',c') = (-1,-1,-1,-1,1,1)$, let
$f' = f + g$ and $g' = -g$.
In any case we have
$$ f'^2 = 0, \ g'^2 = 1, \ f'g' = f' \ \ \mbox{and} \ \ g'f' = -f'. $$
Therefore, replacing $e, f$ and $g$ by $1, f'$ and $g'$, respectively,
$A$ is defined by (\ref{eq:efgf0fg-fe}).

The algebras defined by (\ref{eq:efgf00g00}) and (\ref{eq:efgf0fg-fe}) are not isomorphic, because 
in the algebra defined by (\ref{eq:efgf00g00}), 1 and $-1$ are the only elements $x$
satisfying $x^2 = 1$, but in the algebra defined by (\ref{eq:efgf0fg-fe}), $g^2 = 1$.

In summary, over $\mathbb{C}$
we have exactly 5 non-isomorphic unital algebras 
$U^3_0$, $U^3_1$, $U^3_2$, $U^3_3$ and $U^3_4$
defined by (\ref{eq:efgf00g00}), (\ref{eq:efgf0fg-fe}), (\ref{eq:efg}), (\ref{eq:e000fg0g0}) and (\ref{eq:efgfg0g00}),
respectively.
Over $\mathbb{R}$ we have exactly 6 non-isomorphic unital algebras, 
the algebras defined by the same tables as above
and another algebra $U^3_{2^-}$ defined by (\ref{eq:e000fg0g-f}).

\section{Curled algebras}\label{sec:curled}
In this section we classify curled algebras over $K$ of dimension 3.

Let $C^3_0, C^3_1, C^3_2, C^3_3$ and $C^3_4$ be the algebras on  a base
$\{e, f, g\}$ defined by
\begin{equation}\label{eq:curled-list}
\begin{pmatrix}
0 & 0 & 0 \\
0 & 0 & 0 \\
0 & 0 & 0
\end{pmatrix},
\begin{pmatrix}
0 & 0 & 0 \\
0 & 0 & e \\
0 & -e & 0
\end{pmatrix},
\begin{pmatrix}
0 & 0 & 0 \\
e & f & 0 \\
0 & g & 0
\end{pmatrix},
\begin{pmatrix}
0 & 0 & 0 \\
0 & 0 & 0 \\
e & f & g
\end{pmatrix}
\ \mbox{and} \
\begin{pmatrix}
0 & 0 & e \\
0 & 0 & f \\
0 & 0 & g
\end{pmatrix},
\end{equation} 
respectively.
It is easy to check that these are actually curled (associative) algebras.

We claim that these algebras are not isomorphic to each other.
For an algebra $A$ over $K$, define the square and the left (and right) 
annihilator of $A$ by
$$ A^2 = \left\{ \textstyle{\sum^n_{i=1} x_iy_j \,|\, x_i, y_j \in A, n >0} \right\}, $$
and
$$ {\rm la}(A) = \{x \in A \,|\, xA = 0 \} , \ {\rm ra}(A) = \{x \in A \,|\, Ax = 0 \}, $$
respectively. Obviously these are subalgebras of $A$.
Let
$$ \alpha(A) = \dim_K A^2,\ \beta(A) = \dim_K {\rm la}(A), \ 
\gamma(A) = \dim_K {\rm ra}(A). $$
Clearly, if $A$ and $A'$ are isomorphic, then $\alpha(A) = \alpha(A')$,
$\beta(A) = \beta(A')$ and $\gamma(A) = \gamma(A')$.  
By an easy calculation, we see
$$ \alpha(C^3_0) = 0, \ \alpha(C^3_1) = 1 \ \ \mbox{and} \ \ 
\alpha(C^3_2) = \alpha(C^3_3) = \alpha(C^3_4) = 3, $$
$$ \beta(C^3_2) = 1, \ \beta(C^3_3) = 2 \ \ \mbox{and} \ \ 
\beta(C^3_4) = 0, $$
and
$$ \gamma(C^3_2) = 1, \ \gamma(C^3_3) = 0 \ \ \mbox{and} \ \ 
\gamma(C^3_4) = 2. $$
Thus, our algebras can not be isomorphic to each other, although in this case we need not the
information about the values of $\gamma$.

Next, we shall prove that any curled algebra is isomorphic to one of
the algebras listed in (\ref{eq:curled-list}).
We 
divide cases depending on whether $\{e, f, ef\}$ or $\{e, f, fe\}$
forms a base of $A$ or not.

(a)  Suppose that there are elements $e$ and $f$ of $A$ such that
$\{e, f, ef\}$ is linearly independent.  
Because $A$ is curled,
\begin{equation}\label{eq:e^2}
e^2 = ke \ \ \mbox{and} \ \ f^2 = \ell f 
\end{equation}
for some $k, \ell \in K$.  If $k \neq 0$ (resp. $\ell \neq 0$), by
replacing $e$ (resp. $f$) by $e/k$ (resp. $f/\ell$), we may suppose that
$k$ and $\ell$ are equal to 0 or 1.
Set 
\begin{equation}\label{eq:g}
g = ef,
\end{equation}
then $\{e, f, g\}$ is a linear base of $A$.  
Because $A$ is curled
$$ (e+f)^2 = ke + \ell f + g + fe = m(e+f) $$
for some $m \in K$.  Hence,
\begin{equation}\label{eq:fe}
fe = (m-k)e + (m-\ell)f - g.
\end{equation}
Moreover, we have
\begin{equation}\label{eq:eg}
eg = eef = kef = kg,
\end{equation} 
\begin{equation}\label{eq:gf}
gf = eff = \ell ef = \ell g,
\end{equation}
Therefore,
\begin{equation}\label{eq:ge}
ge = efe = e((m-k)e + (m-\ell)f - g) = k(m-k)e + (m-k-\ell)g
\end{equation}
\begin{equation}\label{eq:fg}
fg = fef = ((m-k)e + (m-\ell)f - g)f = \ell(m-\ell)f + (m-k-\ell)g
\end{equation}
and
\begin{equation}\label{eq:g^2}
g^2 = gef = (k(m-k)e + (m-k-\ell)g)f = ((k+\ell)(m-k) + \ell^2)g.
\end{equation}
Because $e+g$ and
$$ (e+g)^2 = k(m-k+1)e + ((k+\ell)(m-k) + m -\ell^2-\ell)g $$
are linearly dependent, we have
$$ k(m-k+1) = (k+\ell)(m-k) + m -\ell^2-\ell, $$
that is,
$$
(\ell+1)(m - k - \ell) = 0.
$$
Because $\ell \neq -1$, we have $m = k+\ell$.
Thus, (\ref{eq:fe}), (\ref{eq:ge}), (\ref{eq:fg}) and (\ref{eq:g^2}) become
\begin{equation}\label{eq:fe-ge-g^2}
fe = \ell e + k f - g, \ ge = k\ell e, \ fg = k\ell f, 
\ \ \mbox{and} \ \ g^2 = k\ell g.
\end{equation}

(a1)  If $k = \ell = 0$, then
by (\ref{eq:e^2}), (\ref{eq:g}), (\ref{eq:eg}), (\ref{eq:gf}) and (\ref{eq:fe-ge-g^2}) we have
$$ ef = g, fe = -g \ \ \mbox{and} \ \ e^2 = f^2 = eg = gf = ge = fg = g^2 = 0. $$
Let $e' = -g$ and $g' = e$.  Then we have
$e'^2 = e'f = e'g' = fe' = g'e = 0$, $fg' = e'$ and $g'f = -e'$.
Hence, replacing $e$ by $e'$ and $g$ by $g'$, $A$ is defined by
the second table in (\ref{eq:curled-list}), and is isomorphic to  $C^3_1$.

(a2)  If $k = 0$ and $\ell = 1$, then we have
$$  f^2 = f, \ ef = gf = g, \ fe = e-g \ \ \mbox{and} 
\ \ e^2 = eg = ge = fg = g^2 = 0. $$
Let $e' = e-g$, then we have 
$$e'^2 = (e-g)^2 = e^2-eg-ge+g^2 = 0,\ e'f = ef-gf = 0,\ e'g = eg-g^2 = 0$$
and
$$ fe' = fe-fg = e-g = e',\ ge' = ge-g^2 = 0.$$  
Hence, replacing $e$ by $e'$, $A$ is defined by
the third table in (\ref{eq:curled-list}), and is isomorphic to $C^3_2$.

(a3)  If $k = 1$ and $\ell = 0$, then we have 
$$  e^2 = e, \ ef = eg = g, \ fe = f-g\ \ \mbox{and} 
\ \ f^2 = gf = ge = fg = g^2 = 0. $$
Let $e' = g, f' = e, g' = f-g$.  Then, we have
$$ e'^2 = g^2 = 0,\ e'f' = ge = 0,\ e'g' = gf - g^2 = 0,\ f'e' = eg = g = e',$$
$$ f'^2 = e^2 = e = f',\ f'g' = ef-eg = 0,\ g'e' = fg-g^2 = 0,\ g'f' = fe-ge = f-g = g'$$
 and 
$$g'^2 = f^2-fg-gf+g^2 = 0.$$
Thus, $A$ is isomorphic to $C^3_2$ again.

(a4)  Finally, if $k = \ell = 1$, then we have
$$ e^2 = ge = e, \ f^2 = fg = f, \ ef = eg = gf = g^2 = g \ \ 
\mbox{and} \ \ fe = e+f - g. $$
Let $e' = e-g$ and $g' = g-f$.  Then, we have
$$e'^2=e'f=e'g'=fg'=g'e'=g'^2 = 0,\ fe' = e'\ \ \mbox{and}\ \ g'f = g'.$$
Therefore, $A$ is isomorphic to $C^3_2$ in this case too.

(b)  Suppose that $\{e, f, fe\}$ is linearly independent for some 
$e, f \in A$.
As we get (34) above, we can get $ef = k' e+\ell'f-fe$ with
$k', \ell' \in K$.  Thus, $\{e, f, ef\}$ is also linearly independent.
Therefore, this case is included in case (a).

(c)  Suppose that for any linearly independent elements $e$ and $f$ 
of $A$, both $\{e, f, ef\}$ and $\{e, f, fe\}$ are linearly dependent.
Let $\{e, f, g\}$ be a linear base of $A$, 
and let $B$ be the subalgebra of $A$ generated by $\{e, f\}$.  Because
$ef$ and $fe$ 
lie in the space spanned by $\{e, f\}$,
$B$ is spanned by $\{e, f\}$ and is two-dimensional. 

As discussed in Section~\ref{sec:2dim}, there are exactly three non-isomorphic 
curled algebras of dimension 2;  
they are defined by 
\begin{equation}\label{eq:2dim-curled}
\begin{pmatrix}
0 & 0 \\
0 & 0
\end{pmatrix},
\begin{pmatrix}
0 & 0 \\
e & f
\end{pmatrix}
\ \ \mbox{and}\ \ 
\begin{pmatrix}
0 & e \\
0 & f
\end{pmatrix}.
\end{equation}
Hence, we may assume that the subalgebra $B$ is defined by one of the 
tables in (\ref{eq:2dim-curled}).  Because $A$ is curled, 
\begin{equation}\label{eq:g^2=kg} g^2 = kg \end{equation}
for some $k \in K$.  Here we may suppose that $k$ is equal to 0 or 1.

(c1)  Suppose that $B$ is defined by the first table in (\ref{eq:2dim-curled}) 
on the base $\{e,f\}$.
Since $\{e, g, eg\}$ is linearly dependent, $eg = ae+bg$ for some
$a, b \in K$.  Because $e^2 = 0$, we have
$$ 0 = e^2g = e(ae+bg) = b(ae+bg) = abe + b^2g. $$
Hence, $b = 0$ and so $eg = ae$.  By this last equality and (\ref{eq:g^2=kg}) we have
$$ kae = keg = eg^2 = aeg = a^2e. $$
Hence, $a = 0$ or $a = k$.  Similarly, we see that 
$ge = a'e, fg = bf, gf = b'f,$ 
and $a', b$ and $b'$ are equal to $0$ or $k$. 
Thus, if $k = 0$, then $a = a' = b = b' = 0$, and hence, 
$eg = ge = fg = gf = g^2 = 0$.  Hence, $A$ is isomorphic to the zero
algebra $C^3_0$.

Next, suppose that $k = 1$.  Then, $g^2 = g$ and $a,a',b,b' \in \{0,1\}$.
Because $A$ is curled, we have
$$ \ell(e+g) = (e+g)^2 = (a+a')e + g $$
for some $\ell \in K$.  Hence, $a + a' = \ell = 1$. 
Similarly, we have $b + b' = 1$.  So we we have four possibilities
$$ (a,a',b,b') = (0,1,0,1), \ (0,1,1,0), \ (1,0,0,1) \ \ \mbox{and}
\ \ (1,0,1,0). $$
In correspondence to them we have the algebras defined by

\begin{equation}\label{eq:curled-1}
\begin{pmatrix}
0 & 0 & 0 \\
0 & 0 & 0 \\
e & f & g
\end{pmatrix},\  
\begin{pmatrix}
0 & 0 & 0 \\
0 & 0 & f \\
e & 0 & g
\end{pmatrix},\  
\begin{pmatrix}
0 & 0 & e \\
0 & 0 & 0 \\
0 & f & g
\end{pmatrix}
\ \ \mbox{and} \ \ 
\begin{pmatrix}
0 & 0 & e \\
0 & 0 & f \\
0 & 0 & g
\end{pmatrix}.
\end{equation}

The algebra defined by the first table in (\ref{eq:curled-1}) is nothing but $C^3_3$ 
and the algebra defined by the last table is $C^3_4$.
The algebra defined by the second table is isomorphic to
the algebra $C^3_2$ 
via the transformation matrix 
$\begin{pmatrix}
1 & 0 & 0 \\
0 & 0 & 1 \\
0 & 1 & 0
\end{pmatrix}$.
The algebra defined by the third table is also isomorphic
to $C^3_2$ 
via  
$\begin{pmatrix}
0 & 0 & 1 \\
1 & 0 & 0 \\
0 & 1 & 0
\end{pmatrix}$.

(c2)  Suppose that $B$ is defined by the second table in (\ref{eq:2dim-curled}).
As above we have
\begin{equation}\label{eq:eg=ae}
eg = ae,\, ge = a'e,\, a + a' = k\ \ \mbox{with}\ \ 
a, a' \in \{ 0, k \}.
\end{equation} 
Because $\{f, g, fg\}$ is linearly dependent, $fg = bf + cg$ for some
$b, c \in K$.  We have
$$ bf+cg = fg = f^2g = f(bf+cg) = bf + c(bf+cg) = b(c+1)f + c^2g $$
and 
$$ kbf+kcg = kfg = fg^2 = (bf+cg)g = b(bf+cg)+kcg = b^2f+(b+k)cg. $$
Hence, 
\begin{equation}
bc = c(c-1) = b(b-k) = 0.
\end{equation}
Since
$$ 0 = efg = e(bf+cg) = ace $$
and
$$ a'e = a'fe = fge = (bf+cg)e = be + a'ce $$
by (\ref{eq:eg=ae}), we have
\begin{equation}
ac = 0 \ \ \mbox{and} \ \ b = a'(1-c).
\end{equation}

Let $gf = b'f + c'g$ with $b', c' \in K$.  As above we have
\begin{equation}
b'c' = c'(c'-1) = b'(b'-k) = 0.
\end{equation}
and
\begin{equation} ac' = 0 \ \ \mbox{and} \ \ b' = a'(1-c'). \end{equation}
Because $A$ is curled, for any $y , z \in K$, $yf+zg$ and
$$\begin{array}{lll}
(yf+zg)^2 & = & y^2f+kz^2g+yz(bf+cg)+yz(b'f+c'g) \\
& = & (y+(b+b')z)yf + (kz+(c+c')y)zg
\end{array}$$
are linearly dependent.  Hence
$$ (y+(b+b')z)yz = (kz+(c+c')y)yz $$
holds.  Because $y$ and $z$ are arbitrary, it follows that 
\begin{equation}\label{eq:b+b'=k}
b+b' = k \ \ \mbox{and} \ \ c+c' = 1. 
\end{equation}

Now, if $k = 0$, then by (\ref{eq:eg=ae}) -- (\ref{eq:b+b'=k}), we see that
$$ 
a = a' = b = b' = 0 \ \ \mbox{and \ either} \ c = 0, c' = 1  
\ \mbox{or} \ c = 1, c' = 0. $$ 
In the first case, $eg = ge = fg = gg = 0$ and $gf = g$, 
In the second case, $eg = ge = gf = gg = 0$ and $fg = g$.
Hence, $A$ is defined by
\begin{equation}\label{eq:curled-2a}
\begin{pmatrix}
0 & 0 & 0 \\
e & f & 0 \\
0 & g & 0
\end{pmatrix}
\ \ \mbox{or} \ \ 
\begin{pmatrix}
0 & 0 & 0 \\
e & f & g \\ 
0 & 0 & 0
\end{pmatrix}.
\end{equation}
In the first case $A$ is nothing but $C^3_2$.
In the second case, 
$A$ is isomorphic to $C^3_3$
via 
$\begin{pmatrix}
1 & 0 & 0 \\
0 & 0 & 1 \\
0 & 1 & 0
\end{pmatrix}$.

Next, if $k = 1$, then again by (\ref{eq:eg=ae}) -- (\ref{eq:b+b'=k}), we see that
$$ a=0,\, a' = 1 \ \ \mbox{and \ either} \ (b = c' = 0,\, b' = c = 1)\ 
\mbox{or}\ (b = c' = 1,\, b' = c = 0). $$
Hence, $A$ is defined by
\begin{equation}\label{eq:curled-2b}
\begin{pmatrix}
0 & 0 & 0 \\
e & f & g \\
e & f & g
\end{pmatrix}
\ \ \mbox{or} \ \ 
\begin{pmatrix}
0 & 0 & 0 \\
e & f & f \\
e & g & g
\end{pmatrix}.
\end{equation}
The algebra defined by the first table in (\ref{eq:curled-2b}) is isomorphic to $C^3_3$
via 
$\begin{pmatrix}
1 & 0 & 0 \\
0 & -1 & 1 \\
0 & 0 & 1
\end{pmatrix}$.
The algebra defined by the second table in (\ref{eq:curled-2b}) is isomorphic to $C^3_2$ 
via 
$\begin{pmatrix}
1 & 0 & 0 \\
0 & 1 & 0 \\
0 & 1 & -1
\end{pmatrix}$.

(c3)  Suppose that $B$ is defined by the last table.
This case is symmetric with (c2), and from (\ref{eq:curled-2a}) and (\ref{eq:curled-2b}) we have four curled algebras
defined by
\begin{equation}\label{eq:curled-3}
\begin{pmatrix}
0 & e & 0 \\
0 & f & g \\
0 & 0 & 0
\end{pmatrix},\ 
\begin{pmatrix}
0 & e & 0 \\
0 & f & 0 \\
0 & g & 0
\end{pmatrix},\ 
\begin{pmatrix}
0 & e & e \\
0 & f & f \\
0 & g & g
\end{pmatrix}
\ \ \mbox{and}\ \ 
\begin{pmatrix}
0 & e & e \\
0 & f & g \\
0 & f & g
\end{pmatrix}.
\end{equation}
The algebra defined by the first table in (\ref{eq:curled-3}) is isomorphic to $C^3_2$
via 
$\begin{pmatrix}
0 & 0 & 1 \\
0 & 1 & 0 \\
1 & 0 & 0
\end{pmatrix}$.
The algebra defined by the last table
is also isomorphic to $C^3_2$ via the isomorphism 
via 
$\begin{pmatrix}
0 & 0 & 1 \\
0 & 1 & 0 \\
-1 & 1 & 0
\end{pmatrix}$.
The algebras defined by the second and third tables are
isomorphic to $C^3_4$ in the same way as we see above in (c2).

(d)  By the above case classification, 
we have proved that there are exactly five
non-isomorphic curled algebras $C^3_0, C^3_1, C^3_2, C^3_3$
and $C^3_{4}$ defined by the tables in (\ref{eq:curled-list}).

\section{Straight algebras}\label{sec:straight}
Let $A$ be a straight algebra over $K$ of dimension 3, and let
$h$ be an element of $A$ such that
$\{h, h^2, h^3 \}$ forms a linear base of $A$.  Then,
$A$ is commutative, and we can write
\begin{equation}\label{eq:h^4}
h^4 = ah^3 + bh^2 + ch
\end{equation}
with $a, b, c \in K$. 

If $c \neq 0$, let 
$$ e = \frac{1}{c}(h^3 - ah^2 - b h), $$
then $he = eh=h$.  Hence, $e$ is the identity element and $A$ is a
unital algebra. 
Because we already have classified unital algebras in Section~\ref{sec:unital}, 
we suppose that $c = 0$.  Then (\ref{eq:h^4}) becomes
\begin{equation}\label{eq:h^4-new}
h^4 = ah^3+bh^2.
\end{equation}

Here, if $a = b = 0$, then $h^4 = 0$.  Let $e = h, f = h^2$ and $g = h^3$,  then
on the base $\{e,f,g\}$, $A$ is defined by 
\begin{equation}\begin{pmatrix}\label{eq:fg0g00000}
f\ & g & 0 \\
g\ & 0 & 0 \\
0\ & 0 & 0 
\end{pmatrix}.\end{equation}

Next, suppose that $a \neq 0$ and $b = 0$.  Then $h^4 = ah^3$.
Let
$e = \frac{h^3}{a^3},$
then we see that
$\displaystyle \frac{h^n}{a^n} = e $ for all $n \geq 3$.
Let $f' = \frac{h}{a}$ and $g '= f'^2 = \frac{h^2}{a^2}$.
Then, we have
$$ e^2 = ef' = f'e = g'^2 = eg' = g'e = f'g' = g'f'= e. $$
Set $f = f'-e$ and $g = g'-e$.
We thus have
$$ ef= ef'-e^2 = 0,\ eg = eg'-e^2 = 0,\ fe = f'e-e^2 = 0, $$ 
$$ f^2 = f'^2-f'e-ef'+e^2 = g'-e = g,\ fg = f'g'-f'e-eg'+e^2 = 0 $$
and
$$ ge = g'e-e^2 = 0,\ gf = g'f'-g'e-ef'+e^2 = 0,\ g^2 = g'^2-g'e-eg'+e^2 = 0. $$
Hence, $A$ is defined by 
\begin{equation}\begin{pmatrix}\label{eq:eg0}
e\ & 0 & 0 \\
0\ & g & 0 \\
0\ & 0 & 0 
\end{pmatrix}.\end{equation}

Finally, suppose that $b \neq 0$.
Let 
$$g = h^3 - ah^2-bh, $$
then  $hg = gh = 0$ by (\ref{eq:h^4-new}), and so $Ag = gA = 0$.

Let
\begin{equation}\label{eq:e-expression} e = \frac{(a^2 + b)h^2 - a h^3}{b^2}, \end{equation}
then,  by (\ref{eq:h^4-new}) and (\ref{eq:e-expression}) we have
$eh^2 = h^2e = h^2$, $e^2 = e$ and
\begin{equation}\label{eq:(h^2)^2} (h^2)^2 = (a^2+2b)h^2 - b^2e. \end{equation}
Hence, the subspace $B = K\{h^2, h^3\}$ spanned by $h^2$ and $h^3$ 
is a unital straight algebra of dimension 2 over $K$ generated by the
element $h^2$ satisfying (\ref{eq:(h^2)^2}).  
Let $D = a^2 + 4b$.  

If $K = \mathbb{C}$ and $D \neq 0$ or $K = \mathbb{R}$
and $D > 0$, then by the discussion in Section~\ref{sec:2dim}, with
a suitable element $f$ of $B$, $B$ is defined by (\ref{eq:effe}) on
the base $\{e, f\}$.  Therefore, $A$ is defined by
\begin{equation}\label{eq:ef0fe0000}\begin{pmatrix}
e\ & f & 0 \\
f\ & e & 0 \\
0\ & 0 & 0 
\end{pmatrix}.\end{equation}
If $K = \mathbb{R}$ and $D < 0$, then $B$ is defined by (\ref{eq:-e}), and so
$A$ is defined by
\begin{equation}\label{eq:ef0f-e0000}\begin{pmatrix}
e\ & f & 0 \\
f\ & -e & 0 \\
0\ & 0 & 0 
\end{pmatrix}.\end{equation}
If $ D = 0$, $B$ is defined by (\ref{eq:eff0}) and $A$ is defined by
\begin{equation}\label{eq:ef0f00000}\begin{pmatrix}
e\ & f & 0 \\
f\ & 0 & 0 \\
0\ & 0 & 0  
\end{pmatrix}.\end{equation}

We denote the algebras defined by (\ref{eq:fg0g00000}), (\ref{eq:eg0}), (\ref{eq:ef0fe0000}), (\ref{eq:ef0f-e0000}) and (\ref{eq:ef0f00000}) 
by
$S^3_1$, $S^3_2$, $S^3_3$, $S^3_{3^-}$ and $S^3_4$, respectively.
The algebra $S^3_1$ has no nonzero idempotent, 
but the other algebras have one.  Their left annihilators are the 
same $I = Kg$ which is a two-sided ideal.  The residue algebras of 
$S^3_2$, $S^3_3$, $S^3_4$ and $S^3_{3^-}$ 
by $I$ is the two-dimensional algebras defined by
\begin{equation}\label{eq:residue}
\begin{pmatrix}
e & 0  \\
0 & 0
\end{pmatrix},\ 
\begin{pmatrix}
e & f  \\
f & e  
\end{pmatrix},\ 
\begin{pmatrix}
e & f  \\
f & 0   
\end{pmatrix}
\ \mbox{and} \ 
\begin{pmatrix}
e & f  \\
f & -e  
\end{pmatrix}, 
\end{equation}
respectively.  The algebras defined by the first three
tables in (\ref{eq:residue}) are not isomorphic over $\mathbb{C}$ to each other.
Over $\mathbb{R}$, the algebra defined by the last table is not isomorphic
to any other algebras.
Therefore, $S^3_1$, $S^3_2$, $S^3_3$ and $S^3_4$ are 
non-isomorphic over $\mathbb{C}$, and 
$S^3_{3^-}$ is another non-isomorphic algebra over $\mathbb{R}$. 

\section{Waved algebras (strategy)}\label{sec:w-strategy}
In this section and the following two sections we study waved algebras.
Let $A$ be a waved algebra over $K$ of dimension 3 
which is not unital.
Then, there is a non-curled element $f \in A$, such that $\{f, f^2, f^3\}$ 
is linearly dependent, that is, $\{f, f^2\}$ is a linear base of the 
subalgebra $A'$ of $A$ generated by $f$.  

Take $g \in A$ so that $\{f, f^2, g\}$ forms a linear base of $A$.
Here, if $g$ is curled, that is, $g^2 = kg$ for some $k \in K$, then let
$g' = \ell f + g$ with $\ell \in K$.
Write $fg + gf = af + bf^2 + cg$ with $a, b, c \in K$.  
Choose $\ell$ so that $\ell \neq 0$ and $\ell \neq -b$, then $g'$ and
$$g'^2 = \ell^2 f^2 + \ell(fg + gf) + g^2 
= \ell(\ell+b)f^2 + \ell af + (\ell c + k)g $$
are linearly independent, that is, $g'$ is not curled.  
Hence, from the beginning we may assume that $g$ is not curled and 
$\{g, g^2\}$ is a linear base of the subalgebra $A''$ of $A$ generated 
by $g$.  
The subalgebras $A'$ and $A''$ are straight algebras of dimension 2, 
and the intersection $A' \cap A''$ is a subalgebra of $A$ of dimension 1.

From the results in Section~\ref{sec:2dim}, we have exactly 4 straight non-isomorphic
algebras $A^2_3$, $A^2_4$, $A^2_5$ and $A^2_6$ of dimension 2 
over $\mathbb{C}$ defined by
$$
\mbox{(a)}\,
\begin{pmatrix}
0 & 0 \\
0 & e
\end{pmatrix},\ 
\mbox{(b)}\,
\begin{pmatrix}
e & 0 \\
0 & 0
\end{pmatrix},\ 
\mbox{(c)}\,
\begin{pmatrix}
e & f \\
f & e
\end{pmatrix},\ 
\ \mbox{and \ (d)} \ 
\begin{pmatrix}
e & f \\
f & 0
\end{pmatrix},
$$
respectively.
Over $\mathbb{R}$, in addition to the above algebras, 
we have the algebra $A^2_{5^-}$ defined by
$$
(r) \ 
\begin{pmatrix}
e & f \\
f & -e
\end{pmatrix}.
$$

Note that these algebras are all commutative.
As mentioned in Section~\ref{sec:2dim}, 
$A^2_5$ is also defined by (\ref{eq:ef}).

The algebra $A^2_3$ has the unique subalgebra of dimension 1, 
which is the subalgebra $Ke$ generated by $e$ and is isomorphic to 
$A^1_0$.
$A^2_4$ has the unique subalgebra $Kf$ of dimension 1 
isomorphic to $A^1_0$ and the unique subalgebra $Ke$ of dimension 1 
isomorphic to 
$A^1_1$.
From easy calculations
$A^2_5$ has the exactly 3 subalgebras $Ke$, $K(e+f)$ and $K(e-f)$ 
of dimension 1, which are all isomorphic to $A^1_1$.
If we choose the table in (\ref{eq:ef}) for $A^2_5$, $K(e+f)$, $Ke$ and K$f$ are
the three subalgebras isomorphic to $A^1_1$. 
$A^2_6$ has the exactly 2 subalgebras $Ke \cong A^1_1$ and 
$Kf \cong A^1_0$ of dimension 1. 
$A^2_{5^-}$ has the unique subalgebra $Ke \cong A^1_1$ of dimension 1.

Our strategy for classifying waved algebras is that we check through all 
possible combinations of the subalgebras $A'$ and $A''$.

\begin{lemma}\label{lem:eg}
Let $A'$ and $A''$ be straight subalgebras of dimension 2 of 
an algebra $A$ of dimension 3.  Let $\{e, f\}$ and $\{e', g\}$ be 
linear bases of $A'$ and $A''$, respectively.  
Suppose that $\{e,f,g\}$ forms a linear base of $A$ and 
$A' \cap A'' = Ke = Ke'$ holds,
that is, $e = ke'$ for some $k \in K\setminus\{0\}$.  Then we have

(1)  $eg = ge = 0$ if $e'g = 0$. 

(2)  $eg = ge = e$ if $e'g = e'$.

(3)  $eg = ge = kg$ if $e'g = g$.

\end{lemma}
\begin{proof}
(1)  We have
$eg = ke'g = 0$ if $e'g=0$.  Moreover, $ge = kge' = ke'g = 0$ because
$A''$ is commutative from the tables (a), (b), (c), (d) and (r).

(2)  We have 
$eg = ge = ke'g = ke' = e$ if $e'g = e'$.  

(3)  We have
$eg = ge = ke'g = kg$ if $e'g = g$.
\end{proof}

\begin{lemma}\label{lem:e'^2}
Suppose the same situation as in Lemma~\ref{lem:eg}.

(1)  If $e^2 = 0$, then $e'^2 = 0$, and $ef, fe, eg, ge \in Ke$.

(2)  If $e^2 = e$ and $e'^2 = e'$, then $k = 1$ and $e = e'$.
\end{lemma}
\begin{proof}
(1)  Because $e = ke'$ with $k\neq 0$, we see that $e^2 = 0$ if and only if $e'^2 = 0$. 
Because $ef \in A'$ we can write $ef = ae+bf$ with $a, b \in K$.  
If $e^2 = 0$, we have
$$ 0 = e^2f = e(ae+bf) = b(ae+bf) = abe+b^2f. $$
It follows that $b = 0$ and we see $ef \in Ke$.
Similarly, $fe \in Ke$ and $eg, ge \in Ke' = Ke$. 

(2)  If $e^2 = e$ and $e'^2 = e'$, we have 
$$ ke' = e = e^2 = kee' = k^2e'^2 = k^2e'. $$
Because $k \neq 0$, we get $k = 1$ and $e = e'$.
\end{proof}

\begin{corollary}\label{cor:unital}
If $e^2=e$, $e'^2=e'$, $ef=f$ and $e'g = g$, then $A$ is unital.
\end{corollary}
\begin{proof}
Since $A'$ is commutative, $fe = f$.
By  Lemma~\ref{lem:e'^2}, (2) and Lemma~\ref{lem:eg}, (3)  we have $eg = ge = g$. 
These imply that $e$ is the identity element of $A$.
\end{proof}

\begin{lemma}\label{lem:c&c'}
In the same situation as in Lemma~\ref{lem:eg}, let $fg = ae+bf+cg$
and $gf = a'e+b'f+c'g$ with
$a, b, c,a', b', c' \in K$.

(1)  If $f^2 = 0$, then $c = c' = 0$, and moreover, 
if $ef \neq 0$ or $e^2 \neq 0$, 
then $a = a' = 0$.  If $g^2=0$,  then $b=b'=0$, and moreover, 
if $eg \neq 0$ or $e'^2 \neq 0$, then $a = a' = 0$.

(2)  If $f^2=e$ and $eg \in Ke$ (in particular $e^2=0$), then $c=c'=0$
and $eg = aef + be = a'ef + b'e$. 
If $g^2=e'$ and $ef \in Ke$ (in particular $e^2  = 0$), then $b=b'=0$ and 
$ef = kaeg + ce = ka'ef + c'e$.

(3)  If $f^2 = f$, then $c = 0$ or $c = 1$, and $c'=0$ or $c'=1$.  
If $g^2 = g$, then $b = 0$ or $b = 1$, and $b'=0$ or $b'=1$.  
\end{lemma}
\begin{proof}
We have
\begin{equation}\label{eq:f^2g}
f^2g = f(ae+bf+cg)  =  afe + bf^2 + cae + cbf + c^2g. 
\end{equation}

(1)  If $f^2 = 0$, then (\ref{eq:f^2g}) becomes
$$ 0 = afe + cae + cbf + c^2g. $$
Since $afe+cae+cbf \in A'$ and $\{e,f,g\}$ is linearly independent,
we have $c = 0$ and $afe = 0$.  Hence, $a = 0$ if $ef\,(= fe) \neq 0$.
If $ef = 0$ and $e^2 \neq 0$, then $0 = efg = e(ae + bf) = ae^2$ and
we see $a = 0$. Similarly, we have $c' = 0$, and $a' =0$ if $ef \neq 0$ or $e^2 \neq 0$.

Similarly, if $g^2=0$, then $b=b'=0$, and moreover, if $eg \neq 0$  or
$e'^2 \neq 0$, then $a = a' = 0$.

(2)  If $f^2=e$, then (\ref{eq:f^2g}) becomes
$$ eg = afe+(b+ca)e + cbf + c^2g. $$
Here, if $eg \in Ke$ (this holds if $e^2=0$ by Lemma~\ref{lem:e'^2}, (1)), we see $c^2g \in A'$.
Hence, $c=0$ and $eg = aef + be$.  Similarly, we have $c' = 0$ and
$eg = ge = a'ef + b'e$.  The case $g^2=e'$ is similar.

(3) If $f^2=f$, then by (\ref{eq:f^2g}) we have
\begin{equation}\label{eq:ae+bf+cg}
ae+bf+cg = fg = f^2g = afe + ace + b(c+1)f + c^2g. 
\end{equation}
Hence, $c(c-1)g \in A'$ and we have $c = 0$ or $c = 1$.
Similarly, we have $c' = 0$ or $c' = 1$.
If $g^2 = g$, then we have
\begin{equation}\label{eq:ae+bf+cg2}
ae+bf+cg = fg^2 = aeg + abe + c(b+1)g + b^2f,
\end{equation}
and hence $b = 0$ or $b = 1$.
\end{proof}

\begin{lemma}\label{lem:fg=gf}
In the same situation as in Lemma~\ref{lem:eg},
if $ef = 0$ and $e'g = g$ or if $e'g = 0$ and $ef = f$,
then $fg = gf = 0$.
\end{lemma}
\begin{proof}
If $ef = 0$, and $e'g = g$, then by Lemma~\ref{lem:eg}, (3) we have
$  0 = efg = feg = kfg$
and 
$ 0 = gef = kgf. $
Because $k \neq 0$, we see $fg = gf = 0$.
The case where $e'g = 0$ and $ef = f$ is similar. 
\end{proof}

\begin{lemma}\label{lem:g^2=g}
Let $B$ be a straight subalgebra of $A$ of dimension 2 
and let $g$ be an element of $A\setminus B$
such that $Bg = gB = \{0\}$.  Then, $A$ is a straight algebra, if

(1)  $g^2 = g$, or

(2)  $g^2=0$ and $B$ is isomorphic to $A^2_5$, $A^2_6$ or $A^2_{5^-}$.
\end{lemma}
\begin{proof}
Because $B$ is straight, it is defined by (a), (b), (c), (d) or (r) above.

(1)  Suppose that $g^2=g$.  Note that there is an element $x$ in $B$
such that $x^2-x$ and $x^3-x$ are linearly independent.
In fact, $f$, $2e+f$, $2f$, $2e+f$ and $f$ are such elements 
in the algebras defined by (a), (b), (c), (d) and (r), respectively.
For such an element $x$ of $B$, we claim that $x+g$ is not waved.
In fact, if 
$$ a(x+g) + b(x+g)^2 + c(x+g)^3 = ax + bx^2 + cx^3 + (a+b+c)g = 0 $$
for $a, b, c \in K$, then, $a+b+c = 0$ because $g \notin B$.
Hence, $b(x^2-x) + c(x^3-x) = 0$.  Consequently, $a=b=c=0$.
Hence, $A$ is straight.

(2) Suppose that $g^2 = 0$ and $B$ isomorphic to $A^2_5$, $A^2_6$ or $A^2_{5^-}$.  
Note that $B$ has an element $x$ such that $x^2$ and $x^3$ are linearly independent.
In fact, $f$, $e+f$ and $f$ are such elements in the algebras defined by
(c), (d) and (r),  respectively. 
For such an element $x$ of $B$, 
$x+g$, $(x+g)^2 = x^2$ and $(x+g)^3 = x^3$ are linearly independent
because $g \notin B$.  Hence, $A$ is straight.
\end{proof}

\section{Waved algebras (enumeration)}\label{sec:w-enumerate}
Let $A'$ and $A''$ be straight subalgebras of dimension 2 of a waved algebra $A$
such that $A'\cap A''$ is one-dimensional, which is isomorphic to 
either $A^1_0$ or $A^1_1$.
We shall advance the discussion by case division as follows: (a) $A'\cong A^2_3$ and its subdivisions
(aa) $A''\cong A^2_3$, (ab) $A''\cong A^2_4$, (ac) $A''\cong A^2_5$ and (ad) $A''\cong A^2_6$;
(b) $A'\cong A^2_4$ and its subdivisions (bb) $A''\cong A^2_4$, (bc) $A''\cong A^2_5$
and (bd) $A''\cong A^2_6$; (c) $A'\cong A^2_5$ and its subdivisions (cc) $A''\cong A^2_5$
and (cd) $A''\cong A^2_6$; (dd) $A'\cong A''\cong A^2_6$; and finally in the case where $K={\mathbb R}$
(r) $A'\cong A^2_{5^-}$ and its subdivisions (ra) $A''\cong A^2_3$, (rb) $A''\cong A^2_4$, (rc) $A''\cong A^2_5$,
(rd) $A''\cong A^2_6$ and (rr) $A''\cong A^2_{5^-}$.

(a) \ Suppose that $A'$ is isomorphic to $A^2_3$.  Let 
$\{e,f\}$ be a linear base of $A'$ such that $e^2=ef=fe=0$ and $f^2 = e$.

(aa) \ Suppose that $A''$ are also isomorphic to $A^2_3$, and
let $\{e',g\}$ be linear base of $A''$ such that 
$e'^2=e'g=ge'=0$ and $g^2=e'$. 
Because $Ke$ and $Ke'$ are the unique subalgebras of dimension 1
of $A'$ and $A''$, respectively, we see $A' \cap A'' = Ke = Ke'$.  
Hence, $e' = \ell e$ for some $\ell \in K \setminus\{0\}$.

By Lemma~\ref{lem:eg}, (1) we have
$$ eg = ge = 0, $$
and by Lemma~\ref{lem:c&c'}, (2) we see 
$$ fg = ae \ \ \mbox{and} \ \ gf = a'e $$
for some $a, a' \in K$.

Set $g' = a f - g$, then we have
\begin{equation}\label{eq:eg'}
eg' = aef-eg = 0 \ \ \mbox{and} \ \ g'e = afe - ge = 0,
\end{equation}
\begin{equation}
g'^2 = a^2f^2 + g^2 - a(fg+gf) 
= a^2e + \ell e - a(a + a')e = (\ell - aa')e, 
\end{equation}
\begin{equation}
fg' = a f^2 - fg = ae - ae = 0 
\end{equation}
and
\begin{equation}\label{eq:g'f}
g'f = a f^2 - gf = (a - a')e. 
\end{equation}

(aa1)  If $aa' = \ell$ and $a = a'$, then by (\ref{eq:eg'}) -- (\ref{eq:g'f}) we have
$eg' = g'e  = g'^2 = fg' = g'f = 0.$
Thus, replacing $g$ by $g'$, $A$ is defined by
\begin{equation}\label{eq:0e0}
\begin{pmatrix}
0 & 0 & 0 \\
0 & e & 0 \\
0 & 0 & 0
\end{pmatrix}.\end{equation}  

(aa2)  If $aa' = \ell$ and $a \neq a'$, then replacing $g$ by $g'/(a - a')$,
$A$ is defined by
\begin{equation}\begin{pmatrix}\label{eq:e0e0}
0 & 0 & 0 \\
0 & e & 0 \\
0 & e & 0
\end{pmatrix}.\end{equation} 

Now suppose that $aa' \neq \ell$.
 
(aa3)  If $K = \mathbb{C}$, or $K = \mathbb{R}$ and $\ell - aa' > 0$, 
then letting $g'' = g'/\sqrt{\ell-aa'}$, we have
$$ g''^2 = e, \ eg'' = g''e = fg'' = 0 \ \ \mbox{and} \ \ g''f 
= \frac{a - a'}{\sqrt{\ell-aa'}}e. $$
Hence, replacing $g$ by $g''$ and letting $k = (a-a')/\sqrt{\ell-aa'}$, 
$A$ is defined by
\begin{equation}\label{eq:e0kee}
\begin{pmatrix}
0 & 0 & 0 \\
0 & e & 0 \\
0 & ke & e
\end{pmatrix},
\end{equation}
where $k$ can be an arbitrary element of $K$.

(aa4) If $K = \mathbb{R}$ and $\ell - aa' < 0$, then letting
$g'' = g'/\sqrt{aa'-\ell}$, we have
$$ g''^2 = -e, \ eg'' = g''e = fg'' = 0 \ \ \mbox{and} \ \ g''f 
= \frac{a-a'}{\sqrt{aa' - \ell}}e. $$
Hence, replacing $g$ by $g''$ and letting $k = (a-a')/\sqrt{aa' - \ell}$, 
$A$ is defined by
\begin{equation}\label{eq:e0ke-e}
\begin{pmatrix}
0 & 0 & 0 \\
0 & e & 0 \\
0 & ke & -e
\end{pmatrix},
\end{equation}
where $k \in \mathbb{R}$.

(ab)  Suppose that $A''$ is isomorphic to $A^2_4$ and 
let $\{e', g\}$ a linear base of $A''$ such that
$e'^2 = e'g = ge' = 0$ and $g^2 = g$.
Then, $A' \cap A'' = Ke = Ke'$.

By Lemma~\ref{lem:eg}, (1), $ eg = ge = 0,$
and by Lemma~\ref{lem:c&c'}, (2) we have $fg = ae+bf$ with $a, b \in K$ and  
$$ 0 = eg = aef+be = be. $$
Hence, $b = 0$.  Moreover, we have
$$ ae = fg = fg^2 = aeg = 0. $$
Hence, $a = 0$ and so $fg = 0$.  Similarly, we have $gf = 0$.

Thus, $A'g = gA' = \{0\}$ and $g^2 = g$.  Hence, $A$ is not waved by 
Lemma~\ref{lem:g^2=g}, (1), and we exclude this algebra (actually, $A$
is isomorphic to $S^3_2$).

(ac) Suppose that $A''$ is isomorphic to $A^2_5$. However, this case is
impossible because $A'$ has the unique subalgebra of dimension 1 isomorphic to $A^1_0$ but $A''$ has only subalgebras of dimension 1 isomorphic to $A^1_1$.

(ad)  Suppose that $A''$ is isomorphic to $A^2_6$, and
let $\{e', g\}$ be a linear base of $A''$ such that 
$e'^2 = 0$, $e'g = ge' = e'$ and $g^2 = g$.
Because $Ke'$ is the unique subalgebra of $A''$ of dimension 1 
isomorphic to $A^1_0$, we have $Ke = Ke'$. 

By Lemma~\ref{lem:eg}, (2) we have $ eg = ge = e, $
and by Lemma~\ref{lem:c&c'}, (2) we have $fg = ae + bf$ with $a, b \in K$
and
$$ e = eg = aef+be = be, $$
Hence, $b = 1$, that is, $fg = ae+f$. 
We have
$$ ae+f = fg = fg^2 = (ae+f)g = 2ae + f. $$
Hence, $a = 0$ and we see $fg = f$.  Similarly, we have $gf = f.$

These imply that $g$ is the identity element and $A$ is unital, and
we can exclude this algebra.

(b)  Suppose that $A'$ is isomorphic to $A^2_4$. 

(bb)  Suppose that $A''$ is also isomorphic to $A^2_4$.

(bb1)  Suppose that $A' \cap A'' \cong A^1_1$.  
Let $\{e, f\}$ and $\{e',g\}$ be linear bases of $A'$ and $A''$
such that $e^2 = e, ef = fe = f^2 = 0$ and $e'^2 = e', e'g= ge' = g^2 = 0$.
Then, $A' \cap A'' = Ke = Ke'$.  By Lemma~\ref{lem:eg}, (1) we have $eg = ge = 0$.
Moreover, by Lemma~\ref{lem:c&c'}, (1), we see $fg = gf = 0$.
Therefore, $A$ is defined by
\begin{equation}\label{eq:e00}
\begin{pmatrix}
e & 0 & 0 \\
0 & 0 & 0 \\
0 & 0 & 0
\end{pmatrix}\end{equation}.

(bb2)  Suppose that $A' \cap A'' \cong A^1_0$.
Let $\{e, f\}$ and $\{e',g\}$ be linear bases of $A'$ and $A''$
such that $e^2 = ef = fe = 0, f^2 = f$ and $e'^2 = e'g = ge' = 0, g^2 = g$.
Then, $A' \cap A'' = Ke = Ke'$.  
By Lemma~\ref{lem:eg}, (1) we have $ eg = ge = 0. $
Let $fg = ae + bf+ cg$ with $a, b, c \in K$, then by Lemma~\ref{lem:c&c'}, (3) we have
$ c = 0$ or $c = 1$, and $b = 0$ or $b = 1$.
If $c = 0$, then by (\ref{eq:ae+bf+cg}), we have $ae = afe = 0$.  Hence $a = 0$.
If $c = 1$, then by (\ref{eq:ae+bf+cg}) we have $bf = -afe = 0$.  Hence $b = 0$.
Because $ae = aeg = 0$ by (\ref{eq:ae+bf+cg2}), again we have $a = 0$.
Therefore, $fg = bf + cg$.
Similarly, we have $gf = b'f + c'g$ with $b', c' \in K$.

Thus $B = k\{f, g\}$ is a subalgebra of $A$ of dimension 2 such that
$eB = Be = \{0\}$ and $e^2 = 0$.  
Hence, $B$ is curled or isomorphic to $A^2_3$ or $A^2_4$, otherwise
$A$ is straight by Lemma~\ref{lem:g^2=g}, (2).  Because $f^2 = f$ and $g^2 = g$, $B$
cannot be isomorphic to $A^2_3$ nor $A^2_4$, and the only
possibility is that $B \cong  A^2_1$ or $B \cong A^2_2$.
Therefore, choosing a suitable base $\{f, g\}$ of $B$, $A$ can be defined by
\begin{equation}\label{eq:000}
\begin{pmatrix}
0 & 0 & 0 \\
0 & 0& 0 \\
0 & f & g
\end{pmatrix}\ \ 
\mbox{or} \ \ 
\begin{pmatrix}
0 & 0 & 0 \\
0 & 0& f \\
0 & 0& g
\end{pmatrix}.
\end{equation}
 
(bc)  Suppose that 
$A''$ is isomorphic to
$A^2_5$.  Let $\{e, f\}$ be a linear base of $A'$ such that
$e^2 = e$ and $ef = fe = f^2 = 0$.  Then, $A' \cap A'' = Ke \cong A^1_1$.
As observed in the previous section, $A''$ has exactly three subalgebras (1) $Ke$, (2) $K(e+f)$ and (3) $K(e-f)$ 
isomorphic to $A^1_1$ for the table (\ref{eq:effe}). 
Accordingly we have the following three possiblities.

(bc1)  Let $\{e', g\}$ be a linear base of $A''$ such that 
$e'^2 = g^2 = e'$ and $e'g = ge' = g$. 
Suppose that $A' \cap A'' = Ke = Ke'$. By Lemma~\ref{lem:e'^2}, (2) and Lemma~\ref{lem:eg}, (3),  we have $ e' = e$ and $eg = ge = g$.
By Lemma~\ref{lem:fg=gf} we have $fg =fg = 0$. 
Because the subalgebra $A''$ is isomorphic to $A^2_5$, and 
$fA'' = A''f = \{0\}$ and $f^2 = 0$, $A$ is not waved by Lemma~\ref{lem:g^2=g}, (2).

(bc2)  Replacing $e'$ by $\frac{e+f}{2}$ and $g$ by $\frac{e-f}{2}$, let $\{e', g\}$ be a linear base of $A''$ such that 
$e'^2 = e', g^2 = g$ and $e'g = ge' = 0$ as in (\ref{eq:ef}), and suppose that $Ke = Ke'$.  
By Lemma~\ref{lem:eg}, (1) we have $eg = ge = 0,$
and by Lemma~\ref{lem:c&c'}, (1) we see $fg = bf$ and $gf = b'f$ with $b, b' \in K$.
Therefore, $B = K\{f, g\}$ is a subalgebra  of dimension 2.
Because $eB = Be = \{0\}$ and $e^2=e$, $B$ 
must be curled by Lemma~\ref{lem:g^2=g}, (1), because 
otherwise $A$ 
would be straight.  Because $g^2 = g$, $B$ is isomorphic to either $A^2_1$ or $A^2_2$.
Hence, choosing a suitable base $\{f, g\}$ of $B$, $A$ is defined by
\begin{equation}\label{eq:e00-2}
\begin{pmatrix}
e & 0 & 0 \\
0 & 0 & 0 \\
0 & f & g
\end{pmatrix}
\ \ \mbox{or} \ \
\begin{pmatrix}
e & 0 & 0 \\
0 & 0 & f \\
0 & 0 & g
\end{pmatrix} 
\end{equation}

(bc3)  Similarly to (bc2), let $\{e', g\}$ be a linear base of $A''$ such that 
$e'^2 = e', g^2 = g$ and $e'g = ge' = 0$. 
Suppose that $Ke = Kg$.
Because the automorphism of $A''$ interchanging $e'$ and $g$ 
maps the subalgebra $Kg$ onto the subalgebra $Ke'$, 
the situation is essentially the same as (bc2) and we can omit this case.

(bd)  Suppose that 
$A''$ is isomorphic to $A^2_6$.

(bd1)  Suppose that $A' \cap A'' \cong A^1_0$.  
Let $\{e, f\}$ be a linear base of $A'$ such that $e^2=ef=fe=0$ and
$f^2=f$, and $\{e', g\}$ be a linear base of $A''$ such that
$e'^2 = 0, e'g = ge' = e'$ and $g^2 = g$. 
Then, $A'\cap A'' = Ke = Ke'$. 
By Lemma~\ref{lem:eg}, (2) we have $ eg = ge = e.$
Let $fg = ae+bf+cg$ with $a,b,c \in K$.  By Lemma~\ref{lem:c&c'}, (3) we see $b = 0$
or $b=1$.  We have
$$ 0 = efg = e(ae+bf+cg) = ce. $$
Hence, $c=0$.  We have
$$ ae+bf = fg = f^2g = f(ae+bf) = bf. $$
Hence, $a=0$, and we see $fg = bf$, where $b = 0$ or $b=1$.  Similarly,
$gf = b'f$ and $b'=0$ or $b'=1$.  Because
$$ bf = bf^2 = fgf = b'f^2 = b'f, $$
we see $b = b'$.  Therefore, we see $fg = gf = 0$ or $fg = gf = f$.
Consequently, $A$ is defined by
\begin{equation}\label{eq:00e-2}
\begin{pmatrix}
0 & 0 & e \\
0 & f & 0 \\
e & 0 & g
\end{pmatrix}
\ or \ 
\begin{pmatrix}
0 & 0 & e \\
0 & f & f \\
e & f & g
\end{pmatrix}.
\end{equation}

(bd2)  Suppose that $A' \cap A'' \cong A^1_1$.
 Let $\{e, f\}$ be a linear base of $A'$ such that $e^2 = e$ and
$ef = fe = f^2 = 0$.  Let $\{e', g\}$ be a linear base of $A''$ such that 
$e'^2 = e'$, $e'g =ge' = g$ and $g^2 = 0$.  
Then, $Ke = Ke'$.  Then by Lemma~\ref{lem:eg}, (3) and Lemma~\ref{lem:e'^2}, (2) we have
$ eg = ge = g. $
Moreover, by Lemma~\ref{lem:fg=gf}, we see $fg = gf = 0$.
Because the subalgebra $A''$ is isomorphic to $A^2_6$, and
$fA'' = A''f = \{0\}$ and $f^2 = 0$, $A$ is not waved by Lemma~\ref{lem:g^2=g}, (2). 

(c) Suppose that $A'$ is isomorphic to $A^2_5$.

(cc)  Suppose that $A''$ is also isomorphic to $A^2_5$.
As mentioned in (bc), $A''$ has exactly three subalgebras (1) $Ke$, (2) $K(e+f)$ and (3) $K(e-f)$ 
isomorphic to $A^1_1$ for (\ref{eq:effe}). Therefore, $A' \cap A'' \cong A^1_1$ and
we have the three possible combinations: (1) and (1), (1) and (2), and (2) and (2), where (3) is similar to (2) and so can be neglected.
Accordingly, we have the following three cases.

(cc1)  Let $\{e, f\}$ and $\{e', g\}$ be linear bases of $A'$ and $A''$, 
respectively such that $e^2=f^2=e$, $ef=fe=f$, $e'^2=g^2 =e'$ and
$e'g = ge' = g$.
If $A' \cap A'' = Ke = Ke'$, then
$A$ is unital by Corollary~\ref{cor:unital}, and we can exclude this case.

(cc2)  Let $\{e, f\}$ and $\{e', g\}$ be linear bases of $A'$ and $A''$,  
respectively such that $e^2=f^2=e$, $ef=fe=f$, $e'^2=e'$, $g^2=g$ and
$e'g = ge' = 0$.  Suppose that $Ke = Ke'$.
By Lemma~\ref{lem:eg}, (1) we have 
$eg = ge = 0,$ and by Lemma~\ref{lem:fg=gf} we see $fg = gf = 0$.  Thus,
$A$ is not waved by Lemma~\ref{lem:g^2=g}, (1), because the subalgebra $A'$ is straight, 
and $gA' = A'g = \{0\}$ and $g^2 = g$.

(cc3) Let $\{e, f\}$ a linear base of $A'$ with
$e^2 = e, f^2 = f$ and $ef = fe = 0$, and let $\{e', g\}$ be 
a linear base of $A''$ with $e'^2 = e', g^2=g$ and $e'g=ge' = 0$.
Suppose that $A' \cap A'' = Ke = Ke'$,
By Lemma~\ref{lem:eg}, (1) we see $eg = ge = 0.$
Let $fg = ae+bf+cg$ with $a, b, c \in K$.  We have
$$0 = efg = e(ae+bf+cg) = ae, $$
and hence $a = 0$, that is, $fg = bf + cg$.  Similarly, $gf = b'f + c'g$ with $b', c' \in K$.
Hence, $B = \{f,g\}$ is a subalgebra of $A$ of dimension 2 such that $eB = Be = \{0\}$.
Because $e^2 = e$, $B$ must be curled by Lemma~\ref{lem:g^2=g}, (1).  Because $f^2 = f$,
$B$ is isomorphic to $A^2_1$ or $A^2_2$.  Thus, choosing a suitable base $\{f, g\}$
of $B$, $A$ is defined by one of the tables in (\ref{eq:000}), and we get no new algebra here.

(cd)  Suppose that $A''$ is isomorphic to $A^2_6$.  Then, $A' \cap A'' \cong A^1_1$.  
Let $\{e', g\}$ be a linear base of $A''$ such that $e'^2 = e'$,
$e'g =ge' = g$ and $g^2 = 0$.  Because $Ke'$ is the unique subalgebra of dimension 1
isomorphic to $A^1_1$, we see $A' \cap A'' = Ke'$.  

(cd1)  Let $\{e, f\}$ be a linear base of $A'$ such that $e^2=f^2=e$
and $ef = fe = f$, and suppose that $Ke = Ke'$.
Then by Corollary~\ref{cor:unital}, $A$ is unital.

(cd2)  Let $\{e, f\}$ be a linear base of $A'$ such that $e^2=e$, $f^2=f$
and $ef=fe=0$. 
Suppose that $A' \cap A'' = Ke = Ke'$.
By Lemma~\ref{lem:eg}, (3) and Lemma~\ref{lem:e'^2}, (2) we have
$eg = eg = g,$
and by Lemma~\ref{lem:fg=gf} we see $fg = gf = 0$.
But $A$ is not waved by Lemma~\ref{lem:g^2=g}, (1), because
the subalgebra $B = K\{e,g\}$ is straight,
$fB = Bf = \{0\}$ and $f^2 = f$.

The case $A' \cap A'' = Kf$ is symmetric, which we can omit.

(dd)  Suppose that both $A'$ and $A''$ are isomorphic to $A^2_6$.
As observed in the previous section, $A^2_6$ has the exactly 2 subalgebras $Ke\cong A^1_1$
and $Kf\cong A^1_0$. Therefore we have the following two cases.

(dd1)  Suppose that $A' \cap A'' \cong A^1_1$.  
Let $\{e, f\}$ and $\{e', g\}$ be linear bases of $A'$ and
$A''$, respectively such that $e^2 = e, ef = fe = f$
and $f^2 = 0$, and $e'^2 = e', e'g = ge' = g$ and $g^2 = 0$.
Then, $A' \cap A'' = Ke = Ke'$.
By Corollary~\ref{cor:unital}, $A$, is unital and we exclude this case.

(dd2)  Suppose that $A' \cap A'' \cong A^1_0$.
Interchanging $e$ and $f$, and $e'$ and $g$,
let $\{e, f\}$ and $\{e', g\}$ be linear bases of $A'$ and $A''$
such that $e^2=e'^2=0, f^2 = f, ef=fe=e$, $g^2=g$ and $e'g = ge' = e'$. 
Then, $A' \cap A'' = Ke = Ke'$.  Then,
By Lemma~\ref{lem:eg}, (2) we have
$ eg = ge = e. $
Let $fg = ae+bf+cg$ with $a,b,c \in K$.  
We have
$$ e = eg = efg = e(ae+bf+cg) = (b + c)e. $$
It follows that 
\begin{equation}\label{eq:b+c} b + c = 1. \end{equation}
 We have
$$ ae+bf+cg = fg = ffg = f(ae+bf+cg) = ae+bf+c(ae+bf+cg). $$
Hence, 
\begin{equation} a = a(c+1), b = b(c+1) \ \ \mbox{and} \ \  c^2 = c. \end{equation}
Similarly, we have
\begin{equation}\label{eq:a=a(b+1)} a = a(b+1),  c = c(b+1)\ \ \mbox{and} \ \ b^2 = b, \end{equation}
By (\ref{eq:b+c}) -- (\ref{eq:a=a(b+1)}) we see that $a = 0$, and $b = 1, c = 0$ or $b = 0, c = 1$, that is,
$fg = f$ or $fg = g$.  Similarly, we see $gf = f$ or $gf = g$.  Thus, $A$ is defined by

\begin{equation}\label{eq:0ee-4}
\begin{pmatrix}
0 & e & e \\
e & f & f \\
e & f & g
\end{pmatrix},
\ 
\begin{pmatrix}
0 & e & e \\
e & f & f \\
e & g & g
\end{pmatrix},
\ 
\begin{pmatrix}
0 & e & e \\
e & f & g \\
e & f & g
\end{pmatrix}
\ or \ 
\begin{pmatrix}
0 & e & e \\
e & f & g \\
e & g& g
\end{pmatrix}.
\end{equation}

(r)  Finally, suppose that $K = \mathbb{R}$ and $A'$ is isomorphic to
$A^2_{5^-}$.  Let $\{e, f\}$ be a linear base of $A'$ such that
$e^2 = -f^2 = e$ and $ef = fe = f$.  Then $Ke$ is the unique subalgebra
of $A'$ of dimension 1, which is isomorphic to $A^1_1$.

(ra) Suppose that $A'' \cong A^2_3$. However, this case is impossible because $A^2_3$ has the unique subalgebra of dimension 1 isomorphic to $A^1_0$.

(rb) Suppose that $A'' \cong A^2_4$.  Let $\{e', g\}$ be a linear base
of $A''$ such that $e'^2 = e'$ and $e'g = ge' = g^2 = 0$.
Then,
$ A' \cap A'' = Ke = Ke'$.  
We have $eg = ge = 0,$
by Lemma~\ref{lem:eg}, (1) and $fg = gf = 0 $ by Lemma~\ref{lem:fg=gf}.
Because $gA' = A'g = \{0\}$ and $g^2 = 0$,
$A$ is not waved by Lemma~\ref{lem:g^2=g}, (2).

(rc)  Suppose that $A'' \cong A^2_5$.  

(rc1)  Let $\{e', g\}$ be a linear base
of $A''$ such that $e'^2 = g^2 = e'$ and $e'g = ge' = g$.
Suppose that $A' \cap A'' = Ke'$.
By Corollary~\ref{cor:unital} $A$ is unital, and we exclude this case.

(rc2)  Let $\{e', g\}$ be a linear base of $A''$ such that
$e'^2 = e', g^2=g$ and $e'g=ge'=0$.  Suppose that
$A' \cap A'' = Ke = Ke'$.  We have
$eg = eg = 0$ by Lemma~\ref{lem:eg}, (1) 
and $fg = gf = 0$ by Lemma~\ref{lem:fg=gf}.  Since $gA' = A'g' = \{0\}$ and $g^2 = g$, 
$A$ is not waved by Lemma~\ref{lem:g^2=g}, (1)

(rd)  Suppose that $A'' = A^2_6$.  Let $\{e', g\}$ be a linear base of 
$A''$ such that $e'^2 = e', e'g = ge' = g$ and $g^2 = 0$.
Then $A' \cap A'' = Ke = Ke'$.
By Corollary~\ref{cor:unital} $A$ is unital.  
We exclude this case. 

(rr)  Suppose that $A''$ is also isomorphic to $A^2_{5^-}$.
Let $\{e', g\}$ be a linear base of $A''$ such that $e'^2 = -g^2 = e'$ and
$e'g = ge' = g$.  Then, $Ke = Ke'$.
By Lemmas~\ref{lem:eg} and \ref{lem:e'^2} we have
$ eg = ge = g. $
Let $fg = ae+bf+cg$.  We have
$$\begin{array}{lll}
-g & = & -eg = f^2g = f(ae+bf+cg) = af-be+c(ae+bf+cg) \\
& = & (ac-b)e + (bc+a)f + c^2g.
\end{array}$$
Hence, $c^2 = -1$, but there is no such $c$ in $\mathbb{R}$.

\section{Waved algebras (screening)}\label{sec:w-screen}
In the previous section, we have enumerated all possible waved algebras of dimension 3. 
In this section, we shall investigate whether each two of them are isomorphic or not.

(aa1)  Let $W^3_1$ denote the algebra defined by (\ref{eq:0e0}).
Interchanging $f$ and $g$ in the base $\{e, f, g\}$, 
$W^3_1$ is defined by
\begin{equation*}\begin{pmatrix}
0 & 0 & 0 \\
0 & 0 & 0 \\
0 & 0 & e
\end{pmatrix}.\end{equation*}

(aa2)  Next, let $W^3_2$ be the algebra defined by (\ref{eq:e0e0}).
Let $f' = f-g$, then we have
$$ ef' = ef - eg = 0,\, f'e = fe - ge = 0,\, f'^2 = f^2-fg-gf+g^2 = 0, $$
and
$$ f'g = fg - g^2 = 0,\, gf' = gf - g^2 = e.  $$
Hence, replacing $f$ by $f'$, $W^3_2$ is defined by
\begin{equation*}\begin{pmatrix}
0 & 0 & 0 \\
0 & 0 & 0 \\
0 & e & 0
\end{pmatrix}.\end{equation*}

(aa3)  We denote the algebra defined by (\ref{eq:e0kee}) by $W^{3}_3(k)$ for
$k \in K$.
First, we see that $W^3_3(-k)$ is isomorphic to $W^3_3(k)$ via the isomorphism sending 
$e$ to $e$, $f$ to $f$ and $g$ to $-g$.

To prove the converse we may assume that $K = \mathbb{C}$, because if 
$W^3_3(k)$ and $W^3_3(\ell)$ are not isomorphic over $\mathbb{C}$, they are not 
isomorphic over $\mathbb{R}$.
Assume that $\phi: W^3_3(k) \rightarrow W^3_3(\ell)$ is an isomorphism.
Let $ e' = \phi(e), f' = \phi(f)$ and $g' = \phi(g)$.
Because $Ke$ is the left annihilator of  $W^3_3(k)$ (and of $W^3_3(\ell)$),
we see $Ke = Ke'$, that is, 
$$ e' = me $$ for some $m \in K \setminus\{0\}$.  Write
$$ f'= ze + xf + yg \ \ \mbox{and} \ \ g' = we+uf+vg $$ 
with $x, y, z, u, v, w \in K$.  
Then, we have
$$ me = e' = \phi(e) = \phi(f^2) = f'^2 = (x^2+y^2+kxy)e, $$
$$ me = e' = g'^2 = (u^2+v^2+kuv)e $$
$$ 0 = f'g' = (xu+yv+kyu)e $$
and
$$ m\ell e = \ell e' = g'f' = (xu+yv+kxv)e. $$
Thus replacing $x, y, u$ and $v$ by $x/\sqrt{m}$, $y/\sqrt{m}$,
$u/\sqrt{m}$ and $u/\sqrt{m}$, respectively, we have
\begin{equation}\label{eq:x^2+y^2+kxy}
x^2 + y^2 + kxy = 1, \ u^2+v^2+kuv = 1, \ xu+yv+kyu = 0
\end{equation}
and
\begin{equation}\label{eq:xu+uv+kxv}
xu + yv + kxv = \ell.
\end{equation}
Because we have the identity 
$$\begin{array}{lll}
(y^2 - u^2)v & = & u^2v(x^2 + y^2 + kxy -1) + uy(x + ky)(u^2+v^2+kuv - 1) \\
& & \hspace{2.6cm} + (y - uvx - u^2y - kuvy)(xu+yv+kyu), \\
\end{array}$$
we find
\begin{equation*} v(y-u)(y+u) = 0 \end{equation*}
by (\ref{eq:x^2+y^2+kxy})

If $v = 0$, then by (\ref{eq:x^2+y^2+kxy}) we have
$$ u^2 = 1, \ x+ky = 0, \ y^2 =1 \ \ \mbox{and} \ \ x^2 = k^2. $$
Hence, by (\ref{eq:xu+uv+kxv})
$$ \ell = xu = \pm k. $$
If $y = u$, then by (\ref{eq:x^2+y^2+kxy})
$$
(x+v+ku)u = 0.
$$
Here, if $u = 0$, then $y = 0$ and $x^2 = v^2 = 1$ by (\ref{eq:x^2+y^2+kxy}).  Hence, by (\ref{eq:xu+uv+kxv})
$$ \ell = kxv = \pm k. $$
If $x+v+ku = 0$, then
$$ \ell = xu + uv + kxv = -(v+ku)(u+kv) + uv = -k(u^2+v^2+kuv) = -k $$
by (\ref{eq:x^2+y^2+kxy}) and (\ref{eq:xu+uv+kxv}).
If $y = -u$, we can show $ \ell = \pm k$ in a similar manner.

Thus, we have proved that $W^3_3(k)$ and $W^3_3(\ell)$  are isomorphic if and only if
$\ell = \pm k$.

(aa4)  We denote the algebra defined by (\ref{eq:e0ke-e}) by $W^{3}_{3^-}(k)$ for $k \in
\mathbb{R}$.  For any $k \in \mathbb{R}$,
$W^3_{3^-}(k)$ and $W^3_{3^-}(-k)$ are isomorphic via 
the isomorphism sending $e$ to $e$, $f$ to $f$ and $g$ to $-g$.

We shall show that $W^3_{3^-}(k)$ is not isomorphic to $W^3_{3}(k')$ 
for any reals $k$ and $k'$ over $\mathbb{R}$.   
If they were isomorphic, we would have elements 
$f', g' \in W^3_{3^-}(k)$ such that $f'^2 = g'^2 \neq 0$ and $f'g' = 0$.
Let $f' = ze+xf+yg$ and $g' = we +uf+vg$ with $x,y,z,u,v, w \in \mathbb{R}$.
Then we have
\begin{equation}\label{eq:x^2-y^2+kxy}
x^2-y^2+kxy = u^2-v^2+kuv \neq 0 
\ \ \mbox{and} \ \ 
 xu - yv + kyu = 0 
\end{equation}
If $y = 0$, then by (\ref{eq:x^2-y^2+kxy}), $x^2 = u^2 - v^2 + kuv \neq 0$ and $xu = 0$.
Hence, $u = 0$ and 
\begin{equation}\label{eq:x^2=-v^2} x^2 = - v^2 (\neq 0) \end{equation}
The left-hand side of (\ref{eq:x^2=-v^2}) is positive but the right-hand side is negative, a contradiction.
If $y \neq 0$, then by (\ref{eq:x^2-y^2+kxy}) $v = \frac{(x+ky)u}{y}$ and 
$$\begin{array}{lll}
x^2-y^2+kxy & = & u^2 - \frac{(x+ky)^2u^2}{y^2} + \frac{k(x+ky)u^2}{y} \\
& = & -\frac{u^2}{y^2}(x^2-y^2+kxy) \neq 0.
\end{array}$$
This is again impossible, and $W^3_{3^-}(k)$ and $W^3_3(k')$ can not be isomorphic.

On the other hand,
$W^3_{3^-}(k)$ is isomorphic to $W^3_3({\bf i}k)$ over $\mathbb{C}$ via the transformation matrix
$\begin{pmatrix}
1 & 0 & 0 \\
0 & 1 & 0 \\
0 & 0 & -{\bf i}
\end{pmatrix}.$
Hence, for distinct positive numbers $k$ and $k'$,
$W^3_{3^-}(k)$ and $W^3_{3^-}(k')$ are not isomorphic over $\mathbb{C}$,
because $W^3_3({\bf i}k)$ and $W^3_3({\bf i}k')$ are not isomorphic, and
hence they are not isomorphic over $\mathbb{R}$, a fortiori. 

(bb1)  We denote the algebra defined by the table in (\ref{eq:e00}) by
$W^3_4$.

(bb2)  
%
The algebras defined by the first table and the second table in (\ref{eq:000}) are denoted by $W^3_5$ 
and $W^3_6$, respectively.

(bc2)
We denote the algebras defined by the first and the second tables in (\ref{eq:e00-2})  by
$W^3_7$ and $W^3_8$, respectively.

(bd1)  The algebra defined by the first table in (\ref{eq:00e-2}) is not waved.
In fact, $e+f-g$, $(e+f-g)^2 = -2e+f+g$ and $(e+f-g)^3=3e+f-g$
are linearly independent.  The algebra defined by the second table in (\ref{eq:00e-2}) is 
defined by the first one replacing $g$ by $g-f$. 

(dd2)
The algebra defined by the first (resp. last) table in (\ref{eq:0ee-4}), $g$ (resp. $f$) is the identity element.
Hence, they are unital.
Let $W^3_9$ (resp. $W^3_{10}$) be the algebra defined by the second (resp, third) table in (\ref{eq:0ee-4}).
In the second table, let $g' = g - f$, the we have $eg' = g'e = fg' = g'^2 = 0$ and
$g'f = g'$.  Thus, replacing  $g$ by $g'$, $W^3_9$ is defined by
\begin{equation*}\begin{pmatrix}
0 & e & 0 \\
e & f & 0 \\
0 & g & 0
\end{pmatrix}.\end{equation*}
In the third table, let $g' = g - f$, then we have $eg' = g'e = g'f = g'^2 = 0$ and
$fg' = g'$.  Replacing  $g$ by $g'$, $W^3_{10}$ is defined by
\begin{equation*}\begin{pmatrix}
0 & e & 0 \\
e & f & g \\
0 & 0 & 0
\end{pmatrix}.\end{equation*}

In summary, any non-unital waved algebra $A$ of dimension 3 over 
$\mathbb{C}$ is isomorphic to $W^3_1$, $W^3_2$, $W^3_4$, $W^3_5$,
$W^3_6$, $W^3_7$, $W^3_8$, $W^3_9$, $W^3_{10}$ or $W^3_3(k)$ with 
$k \in \mathcal{H}
= \{x+y{\bf i} \in \mathbb{C} \,|\, x>0 \ \mbox{or}\ x=0, y\geq 0\}$.
Over $\mathbb{R}$, $A$ can be isomorphic to $W^3_{3^-}(k)$ with
$k \geq 0$ other than above.  We shall show these algebras are not 
isomorphic to each other.  As we already showed that $W^3_{3}(k)$ and
$W^3_{3}(k')$ are not isomorphic for distinct $k, k' \in \mathcal{H}$, 
 $W^3_{3^-}(k)$ and $W^3_{3^-}(k')$ are not
isomorphic over $\mathbb{R}$ for any distinct
nonnegative $k$ and $k'$.  Moreover, $W^3_{3}(k)$ and $W^3_{3^-}(k')$
are not isomorphic over $\mathbb{R}$ for any nonnegative $k$ and $k'$.

We use the same notations as in Section~\ref{sec:curled} for an algebra $A$ over $K$;
$$
{\rm la}(A) = \{x \in A \,|\, xA = 0 \}, \ 
{\rm ra}(A) = \{x \in A \,|\, Ax = 0 \}, $$
and
$$ \alpha(A) = \dim_K A^2,\ \beta(A) = \dim_K {\rm la}(A),\  
\gamma(A) = \dim_K {\rm ra}(A). $$

We see
$$ \alpha(W^3_1) = \alpha(W^3_2) = \alpha(W^3_3(k)) = \alpha(W^3_{3^-}(k))
= \alpha(W^3_4) = 1, $$
$$ \alpha(W^3_5) = \alpha(W^3_6) = 2 \ \ \mbox{and} \ \ \alpha(W^3_7)
= \alpha(W^3_8) = \alpha(W^3_9) = \alpha(W^3_{10}) = 3. $$
Moreover,
$$ \beta(W^3_1) = \beta(W^3_2) = \beta(W^3_4) = \beta(W^3_5) = 2, $$
$$ \beta(W^3_3(k)) = \beta(W^3_{3^-(k)}) = \beta(W^3_6) = \beta(W^3_7) 
= \beta(W^3_{10}) = 1, \ \beta(W^3_8) = \beta(W^3_9) = 0,$$
$$ \gamma(W^3_1) = \gamma(W^3_2) = \gamma(W^3_4) = \gamma(W^3_6) = 2, $$
and
$$ \gamma(W^3_3(k)) = \gamma(W^3_{3^-}(k)) = \gamma(W^3_5) = \gamma(W^3_8) 
= \gamma(W^3_9) = 1, \ \gamma(W^3_7) = \gamma(W^3_{10}) = 0. $$
These values are summarized in the table shown below.
\begin{center}
\begin{tabular}{|c||c|c|c|c|c|c|c|c|c|c|c|} \hline
            & $W^3_1$ & $W^3_2$ &$W^3_3(k)$ &$W^3_4$ &$W^3_5$ &$W^3_6$ &$W^3_7$ &$W^3_8$ &$W^3_9$ &$W^3_{10}$ &$W^3_{3^-}(k)$ \\ \hline\hline
$\alpha$ & 1&1&1&1&2&2&3&3&3&3&1\\ \hline
$\beta$ & 2&2&1&2&2&1&1&0&0&1&1\\ \hline
$\gamma$ & 2&2&1&2&1&2&0&1&1&0&1\\ \hline
\end{tabular}
\end{center}
 
This table implies that $W^3_3(k)$, $W^3_{3^-}(k)$, $W^3_5$ and $W^3_6$ 
are not isomorphic to any of the others. 
$W^3_1$, $W^3_2$ and $W^3_4$ are not isomorphic to any algebra 
in the other group.  $W^3_2$ is not isomorphic to $W^3_1$ nor to $W^3_4$
because $W^3_1$ and $W^3_4$ are commutative but $W^3_2$ is not.
$W^3_1$ and $W^3_4$ are not isomorphic because $W^3_4$ has 
a nonzero idempotent $e$;
$e^2 = e$ but $W^3_1$ has no nonzero idempotent.
Finally, $W^3_7$, $W^3_8$, $W^3_9$ and $W^3_{10}$ are not
isomorphic to any algebra in the other group.
$W^3_9$ and $W^3_{10}$ have two zeropotent elements $e$ and $g$;
$e^2 = g^2 = 0$ but $W^3_7$ and $W^3_8$ have no two zeropotent elements
that are linearly independent.  Hence, these algebras are not isomorphic to each other.

We have proved that all the algebras are not isomorphic to each other.

\section{Summary}\label{sec:summary}
By summarizing the results in Sections~\ref{sec:unital}, \ref{sec:curled}, \ref{sec:straight} and \ref{sec:w-screen}, we list up all the multiplication tables of
three-dimensional algebras. It is easy to see that all the tables satisfy associativity, that is, $(e_ie_j)e_k=e_i(e_je_k)$ for all $i,j,k\in\{1,2,3\}$ where $e_1=e$, $e_2=f$ and $e_3=g$.

Over $\mathbb{C}$, we have, up to isomorphism, 
exactly 5 unital algebras $U^3_0, U^3_1, U^3_2, U^3_3, U^3_4$ defined by
$$
\begin{pmatrix}
e & f & g \\
f & 0 & 0 \\
g & 0 & 0
\end{pmatrix}, \ 
\begin{pmatrix}
e & f & g \\
f & 0 & f \\
g & -f & e
\end{pmatrix}, \ 
\begin{pmatrix}
e & 0 & 0\\
0 & f & 0\\
0 & 0 & g
\end{pmatrix}, \ 
\begin{pmatrix}
e & 0 & 0\\
0 & f & g\\
0 & g & 0
\end{pmatrix}, \ 
\begin{pmatrix}
e & f & g \\
f & g & 0 \\
g & 0 & 0
\end{pmatrix},
$$
respectively, exactly 5 non-unital curled algebras $C^3_0, C^3_1, C^3_2, C^3_3, C^3_4$
defined by
$$
\begin{pmatrix}
0 & 0 & 0 \\
0 & 0 & 0 \\
0 & 0 & 0
\end{pmatrix}, \ 
\begin{pmatrix}
0 & 0 & 0 \\
0 & 0 & e \\
0 & -e & 0
\end{pmatrix}, \
\begin{pmatrix}
0 & 0 & 0 \\
e & f & 0 \\
0 & g & 0
\end{pmatrix}, \ 
\begin{pmatrix}
0 & 0 & 0 \\
0 & 0 & 0 \\
e & f & g
\end{pmatrix}, \
\begin{pmatrix}
0 & 0 & e \\
0 & 0 & f \\
0 & 0 & g
\end{pmatrix},
$$ 
respectively, exactly 4 non-unital straight algebras $S^3_1, S^3_2, S^3_3,
S^3_4$ defined by
$$
\begin{pmatrix}
f\ & g & 0 \\
g\ & 0 & 0 \\
0\ & 0 & 0 
\end{pmatrix},\ 
\begin{pmatrix}
e\ & 0 & 0 \\
0\ & g & 0 \\
0\ & 0 & 0 
\end{pmatrix},\ 
\begin{pmatrix}
e\ & f & 0 \\
f\ & e & 0 \\
0\ & 0 & 0 
\end{pmatrix},\ 
\begin{pmatrix}
e\ & f & 0 \\
f\ & 0 & 0 \\
0\ & 0 & 0  
\end{pmatrix},
$$ 
respectively, and exactly 9 non-unital waved algebras $W^3_1$, $W^3_2$, 
$W^3_4$, $W^3_5$, $W^3_6$, $W^3_7$, $W^3_8$, $W^3_9$, $W^3_{10}$ 
defined by
$$
\begin{pmatrix}
0 & 0 & 0 \\
0 & 0 & 0 \\
0 & 0 & e
\end{pmatrix},\ 
\begin{pmatrix}
0 & 0 & 0 \\
0 & 0 & 0 \\
0 & e & 0
\end{pmatrix},\ 
\begin{pmatrix}
e & 0 & 0 \\
0 & 0 & 0 \\
0 & 0 & 0
\end{pmatrix},\ 
\begin{pmatrix}
0 & 0 & 0 \\
0 & 0 & 0 \\
0 & f & g
\end{pmatrix},\ 
\begin{pmatrix}
0 & 0 & 0 \\
0 & 0 & f \\
0 & 0 & g
\end{pmatrix},
$$
$$
\begin{pmatrix}
e & 0 & 0 \\
0 & 0 & 0 \\
0 & f & g
\end{pmatrix},\ 
\begin{pmatrix}
e & 0 & 0 \\
0 & 0 & f \\
0 & 0 & g
\end{pmatrix},\ 
\begin{pmatrix}
0 & e & 0 \\
e & f & 0 \\
0 & g & 0
\end{pmatrix},\ 
\begin{pmatrix}
0 & e & 0 \\
e & f & g \\
0 & 0 & 0
\end{pmatrix},
$$
respectively and one infinite family $\big\{W^3_3(k)\big\}_
{k \in \{x+y{\bf i} \,|\, x>0 \ \mbox{or}\ x=0, y\geq 0\}}$
of non-unital waved algebras defined by
$$
\begin{pmatrix}
0 & 0 & 0 \\
0 & e & 0 \\
0 & ke & e
\end{pmatrix}.
$$

Over $\mathbb{R}$, in addition to the above algebras, we have one unital algebra 
$U^3_{2^-}$, one non-unital straight algebra $S^3_{3^-}$ and one infinite family 
$\big\{W^3_{3^-}(k)\big\}_{k \geq 0}$ of non-unital waved algebras
defined by
$$
\begin{pmatrix}
e & 0 & 0\\
0 & f & g\\
0 & g & -f
\end{pmatrix},\ 
\begin{pmatrix}
e\ & f & 0 \\
f\ & -e & 0 \\
0\ & 0 & 0 
\end{pmatrix},\   
\begin{pmatrix}
0 & 0 & 0 \\
0 & e & 0 \\
0 & ke & -e
\end{pmatrix},
$$
respectively.

We remark that zeropotent algebras are only $C^3_0$ and $C^3_1$ which are alternative matrices.

An algebra is {\it indecomposable}, if it is not isomorphic to 
a direct sum of nontrivial subalgebras. We will not go into the details, but it is possible to show that
the algebras $U^3_0, U^3_1, U^3_4, C^3_1, C^3_2, C^3_3, C^3_4, S^3_1, 
W^3_2, W^3_3(k)$, $W^3_{3^-}(k)$, $W^3_9, W^3_{10}$ are indecomposable, 
and the others are not.

Peirce \cite{P} listed five families of "pure" algebras of dimension 3.  They correspond
to our $U_4, S_1, W_3(k), W_2$ and $C_1$.  The list of unilal algebras
of dimension 3 given by Scheffer \cite{Sch} and 
Study \cite{St} is in accordance with our list.

\bibliographystyle{amsplain}

\noindent
Y. Kobayashi, K. Shirayanagi, and M. Tsukada\\
Department of Information Science, Toho University, Miyama 2-2-1 Funabashi, Chiba 274-8510, Japan\\ kobayasi@is.sci.toho-u.ac.jp, kiyoshi.shirayanagi@is.sci.toho-u.ac.jp, and tsukada@is.sci.toho-u.ac.jp\\

\noindent
S.-E. Takahasi\\
Laboratory of Mathematics and Games, Katsushika 2-371 Funabashi, Chiba 273-0032 Japan\\
sin\_ei1@yahoo.co.jp

\end{document}